\documentclass[a4paper,10pt]{amsart}

\usepackage{a4wide}
\usepackage{float}
\usepackage{euscript,eufrak,verbatim}
\usepackage{graphicx}
\usepackage[usenames]{color}
\usepackage[colorlinks,linkcolor=red,anchorcolor=blue,citecolor=blue]{hyperref}
\usepackage{amsmath}
\usepackage{amsthm}
\usepackage[all]{xy}
\usepackage{graphicx} 
\usepackage{amssymb}

\newtheorem{theomain}{Theorem}

\newtheorem{thm}{Theorem}[section]
\newtheorem{prop}[thm]{Proposition}

\newtheorem{lem}[thm]{Lemma}
\newtheorem{cor}[thm]{Corollary}

\newtheorem{ques}[thm]{Question}

\newtheorem{ex}[thm]{Example}

\newtheorem{rem}[thm]{Remark}

\def\N{\mathbb{N}}
\def\Z{\mathbb{Z}}
\def\N{\mathbb{N}}
\def\R{\mathbb{R}}

\def\F{\mathbb{F}}
\def\T{\mathbb{T}}
\def\XX{\mathcal{X}}
\def\FF{\mathcal{F}}

\def\SS{\mathcal{S}}

\def\mdim{\text{\rm mdim}}
\def\supp{\text{\rm supp}}
\def\stab{\text{\rm stabdim}}
\def\Id{\text{\rm Id}}
\def\diam{\text{\rm diam}}
\def\AA{\mathcal{A}}
\def\BB{\mathcal{B}}

\def\Id{\text{\rm Id}}

\makeatletter 
\@addtoreset{equation}{section}
\numberwithin{equation}{section}

\makeatletter
\@namedef{subjclassname@2020}{\textup{2020} Mathematics Subject Classification}
\makeatother

\title{Mean dimension of continuous cellular automata}

	\author{David Burguet}
	\address
	{Sorbonne Universite, LPSM, 75005 Paris, France}
	\email{david.burguet@upmc.fr}
	\author{Ruxi Shi}
\address
{Institute of Mathematics, Polish Academy of Sciences, ul. \'Sniadeckich 8, 00-656 Warszawa, Poland}
\email{rshi@impan.pl}

\begin{document}

	\maketitle

	\begin{abstract}

 We investigate the mean dimension of a cellular automaton (CA for short) with a compact non-discrete space of states. 
 A formula for the mean dimension  is established for\textit{ (near) strongly permutative}, \textit{permutative algebraic} and \textit{unit} one-dimensional automata. In higher dimensions,  a CA permutative algebraic or having a spaceship has infinite mean  dimension. However, building on Meyerovitch's example \cite{meyerovitch2008finite}, we give an example of algebraic surjective cellular automaton with positive finite mean dimension. 
	\end{abstract}

\tableofcontents
	
\section{Introduction}	
	The most basic topological invariant of topological dynamical systems is the topological entropy which has been studied for a long time. Gromov \cite{G} introduced a new topological invariant of dynamical systems called {\it mean topological dimension} 
	as a dynamical analogue of topological covering dimension. It was further developed by Lindenstrauss and Weiss \cite{LindenstraussWeiss2000MeanTopologicalDimension} with their work on  the {\it metric mean dimension}, the dynamical quantity associated to the Minkowski dimension. A dynamical system of finite entropy or finite dimension has zero mean dimension.
	
	Let $X$ be a compact metric space and let $f$ be a continuous function given by  $f:X^I\rightarrow X$ with a finite set $I \subset \Z^d$. A {\it cellular automaton (CA for short)}  on $X^{\Z^d}$ with local rule $f$ is the continuous map $F: X^{\Z^d} \to X^{\Z^d}$ defined by
	$$
	F\left((x_n)_{n\in \Z}\right)=\left(f((x_{n+j})_{j\in I})\right)_{n\in \Z^d}.
	$$  
	We also denote sometimes the cellular automaton  $F$ associated to the local rule $f$ by $T_f$.
	The space $X$ is called the {\it set of states} of the cellular automaton $T_f$. When  $X$ is finite, the topological entropy of a one-dimensional cellular automaton  (i.e. $d=1$)  is known to be finite. Moreover, the explicit value of the topological entropy in some cases may be computed. For example, when $d$ is equal to $1$, the set of states $X$ is the finite field $\F_p$ with a prime number $p$ and the local rule $f$ is linear, i.e. $f\left((x_j)_{j\in I}\right)=\sum_{j\in I} a_j x_j$ for some $a_j\in \F_p^*$, the topological entropy of the cellular automaton $T_f$ on $\F_p^\Z$ is equal to $\diam(I\cup\{0\}) \cdot \log p$
	where $\diam(I\cup\{0\})$ stands for the diameter of $I\cup\{0\}$ with respect to the Euclidean distance on $\R$ \cite{ward2000additive}. For $d>1$, the topological entropy of a linear cellular automaton $F$ is always infinite unless $F=\pm \Id$  \cite{morris1998entropy,lakshtanov2004criterion}. An example of a multidimensional surjective cellular automaton with nonzero finite entropy  was constructed in \cite{meyerovitch2008finite}.\\

	In this paper, we investigate  cellular automata with a non-discrete  compact finite  dimensional set of states $X$. 
	We involve the theory of mean dimension to study these cellular automata. We show in Proposition \ref{prop:upper bound} that the mean dimension of a one-dimensional cellular automaton is bounded from above by $\diam(I\cup\{0\}) \cdot\stab X$, where $\stab X$ denotes the stable topological dimension of $X$. As a consequence, the mean dimension of a one-dimensional automaton is finite. Then one may wonder whether this upper bound is the value of its mean dimension. The {\it permutative} property of  cellular automaton implies that it has a topological factor with mean dimension  $\diam(I\cup\{0\}) \cdot\stab X$.
	However, the permutative property does not ensure that such an upper bound is achieved as the  value of its mean dimension. Several examples are discussed in Section \ref{sec:Examples with a lower dimensional subsystem}. Nevertheless, when a one-dimensional cellular automaton is {\it (near) strongly permutative} or permutative {\it algebraic},  the mean dimension is equal to $\diam(I\cup\{0\}) \cdot\stab X$  (Theorem \ref{strong}, Corollary \ref{near} and Lemma \ref{lem:algebraic}). We summarize  the aforementioned results in the following statement :
	\begin{theomain}
Let $F:X^\Z\circlearrowleft$ be a one-dimensional CA with local rule $f:X^I\rightarrow X$. Then we have 
\begin{equation}\label{main}
\mdim(F) \leq \diam(I\cup\{0\}) \cdot\stab X.
\end{equation}	
The equality holds in (\ref{main}) when $F$ is near strongly permutative or permutative algebraic, but the inequality may be strict for a general permutative CA. 
	\end{theomain}

	On the other hand, when considering a unit (i.e. $I=\{1\}$) one-dimensional cellular automaton $(X, T)$, we show that its natural extension is topologically conjugate to a full shift and the mean dimension of it is then given by that of its natural extension. To this end, we also investigate the relation between a dynamical system and its natural extension: whereas the topological entropy is preserved by natural extension, the mean dimension of a dynamical system is always larger than or equal to that of its natural extension (Proposition \ref{nat}) but may differ.  In Proposition \ref{prop:example natural} we give an example of a topological system, whose   mean dimension  is strictly larger than that of its natural extension.



	Furthermore, we study multidimensional cellular automata with a compact non-discrete set of states. We prove that nontrivial multidimensional algebraic permutative cellular automata have infinite mean dimension (Lemma \ref{lem:multidimension algebraic}). 	Building on Meyerovitch's example we construct a multidimensional algebraic non-permutative  surjective cellular automaton with nonzero finite mean dimension (Proposition \ref{prop:example finite nonzero higher dimension}). Beyond the algebraic  property,	we also show that if a  multidimensional cellular automaton has a {\it spaceship} (see the definition in Section \ref{sec:multi}), then its mean dimension is infinite.  
	
	\begin{theomain}
Let $F:X^{\Z^d}\circlearrowleft$ be a CA with $d>1$.  When $F$ is a nontrivial  permutative algebraic CA or $F$ has a spaceship,  the mean dimension of $F$ is infinite. But there are 
algebraic non-permutative $F$ surjective cellular automata with nonzero finite mean dimension.
\end{theomain}

\section{Background on mean dimension}	\label{sec:background}

	Let $X$ be a compact space. For two finite open covers $\mathcal{A}$ and $\mathcal{B}$ of $X$, we say that the cover $\mathcal{B}$ is \textit{finer} than the cover $\mathcal{A}$, and write $\mathcal{B}\succ \mathcal{A}$, if for every element of $\mathcal{B}$, one can find an element of $\mathcal{A}$ which contains it.  For a finite open cover $\mathcal{A}$  of $X$, we define the quantities 
	$$
	\text{\rm ord}(\mathcal{A}):=\sup_{x\in X} \sum_{A\in \mathcal{A}} 1_A(x)-1,
	$$
	and
	$$
	D(\mathcal{A}):=\min_{\mathcal{B}\succ \mathcal{A}} \text{ord}(\mathcal{B}).
	$$
	Clearly, if $\mathcal{B}\succ \mathcal{A}$ then $D(\mathcal{B})\ge D(\mathcal{A})$. The \textit{(topological) dimension} of $X$ is defined by
	$$
	\text{dim}(X):=\sup_{\mathcal{A}} D(\mathcal{A}),
	$$
	where $\mathcal{A}$ runs over all finite open covers of $X$. 	For a non-empty compact  $X$,  the {\it stable topological dimension}
	of $X$ is given  by
	$$
	\stab(X):=\lim_{n\to \infty} \frac{\dim(X^n)}{n}=\inf_{n\to \infty} \frac{\dim(X^n)}{n}.
	$$
	The limit above exists by sub-additivity of the sequence $\{\dim(X^n) \}_{n\ge 1}$. Moreover, if $X$ is finite dimensional, then we have  either $\stab(X)=\dim(X)$ (the set $X$ is then said of {\it basic type}) or $\stab(X)=\dim(X)-1$ (and $X$ is said of  {\it exceptional type}).

	 For finite open covers $\mathcal{A}$ and $\BB$, we set the joint $\AA \vee \BB:=\{U\cap V: U\in \AA, V\in \BB \}$. It is easy to check that $D(\AA\vee \BB)\le D(\AA)+D(\BB)$. 
	Let $(X, T)$ be a topological  dynamical system, i.e. $X$ is a compact metrizable space and $T:X\circlearrowleft$ is a continuous map. The {\it mean dimension} of $(X, T)$ is defined by
	$$
	\mdim(X, T)=\sup_{\alpha} \lim_{n\to \infty} \frac{D(\bigvee_{i=0}^{n-1} T^{-i}\alpha )}{n},
	$$
	where $\alpha$ runs over all finite open covers of $X$. The existence of the limit  follows from  the sub-additivity of the sequence $\left(D(\bigvee_{i=0}^{n-1} T^{-i}\alpha)\right)_n$. We write also $\mdim(X,T,\alpha)=\lim_{n\to \infty} \frac{D(\bigvee_{i=0}^{n-1} T^{-i}\alpha )}{n}.$

	For a set $Z$ and $\epsilon>0$, a map $f:X\to Z$ is called \textit{$(\rho, \epsilon)$-injective} if diam$(f^{-1}(z))<\epsilon$ for all $z\in Z$. The \textit{$\epsilon$-dimension} $\text{dim}_\epsilon (X, \rho)$ is defined by 
	$$
	\text{dim}_\epsilon (X, \rho)=\inf_Y \text{dim}(Y),
	$$
	where $Y$ runs over all compact metrizable spaces for which there exists a $(\rho, \epsilon)$-injective continuous map $f:X\to Y$.

	We mention some basic properties of mean dimension. We refer to the book \cite{Coo05} for the proofs and further properties. 
	\begin{itemize}
		\item 
	
		$$
		\mdim(X,T)=\sup_{\epsilon} \lim_{n\to \infty} \frac{\dim_{\epsilon}(X, d_n)}{n},
		$$
		where $d$ is a metric on $X$ compatible with the topology and $d_n(x,y)=\max_{0\le i\le n-1}d(T^ix, T^iy)$. We write sometimes $\mdim(X,T,d,\epsilon)=\lim_{n\to \infty} \frac{\dim_{\epsilon}(X, d_n)}{n}.$
		\item If $(Y,T)$ is a subsystem of $(X,T)$, i.e. $Y$ is a closed $T$-invariant subset of $X$, then $\mdim(Y,T)\le \mdim(X,T)$.
		\item For $n\in \N$, $\mdim(X,T^n)=n\cdot \mdim(X,T)$.
		\item For dynamical systems $(X_i,T_i)$, $1\le i\le n$, we have
		$$
		\mdim(X_1\times X_2 \times \dots \times X_n, T_1\times T_2 \times \dots \times T_n)\le \sum_{i=1}^n \mdim(X_i, T_i).
		$$
	\end{itemize}	
	
	Let $(X, d, T)$ be a topological dynamical system where $d$ is a metric compatible with the topology of $X$. Let $K\subset X$ and $\epsilon>0$. A subset $E$ of $X$ is said to be {\it $(n, \epsilon)$-separated}  if we have  $d_n(x,y)>\epsilon$ for any $x\neq y\in E$. Denote by $s_n(d, T, K, \epsilon)$  the largest cardinality of any $(n, \epsilon)$-separated subset of $K$. Define
	$$
	h_d(K, T, \epsilon)=\limsup_{n\to \infty} \frac{1}{n} \log s_n(d, T, K, \epsilon).
	$$
	We sometimes write $h_d(T, \epsilon)$ when $K=X$. The topological entropy is the limit of  $h_d(T, \epsilon)$ when $\epsilon$ goes to zero. 
	The {\it upper metric mean dimension} of the system $(\XX, d, T)$ is defined by
	$$
	\overline{\mdim}_M (X, d, T)=\limsup_{\epsilon\to 0} \frac{h_d(T, \epsilon)}{\log \frac{1}{\epsilon}}.
	$$
	Similarly, the {\it lower metric mean dimension} is defined by
	$$
	\underline{\mdim}_M (X, d, T)=\liminf_{\epsilon\to 0} \frac{h_d(T, \epsilon)}{\log \frac{1}{\epsilon}}.
	$$
	If the upper and lower metric mean dimensions coincide, then we call their common value the metric mean dimension of $(X, d, T)$ and denote it by ${\mdim}_M (X, d, T)$. Unlike the topological entropy, the metric mean dimension depends on the metric $d$. An important property of metric mean dimension is
	$$
	\mdim(X,T)\le \underline{\mdim}_M (X, d, T),
	$$
	for any metric $d$ \cite{LindenstraussWeiss2000MeanTopologicalDimension}.
	
	\section{Natural extension $(\widetilde{X_T}, \widetilde{T})$ and skew-product}

Let $(X_i)_{i\in \mathbb N}$ be a sequence of compact metric spaces and let $\mathbf{f}= (f_i)_{i\in \mathbb N}$ be a sequence of continuous maps   $f_i:X_{i+1}\rightarrow X_{i}$.  Then the inverse limit of $\mathbf{f}$ denoted by $\varprojlim \mathbf {f}$ is the set of sequences $(x_i)_{i\in \N}\in \prod_{i\in \N} X_i$ satisfying $f_i(x_{i+1})=x_i$ for all $i\in \mathbb N$. The maps $f_i$, $i\in \mathbb N$ are called the {\it bonding maps} of the inverse limit.

	Let $(X,T)$ be a topological dynamical system. The inverse limit $\varprojlim \mathbf {f}$ with $f_i=T$ and $X_i=\bigcap_{n\ge 0}T^n X$ for all $i\in \mathbb N$ is denoted by $\widetilde{X_T}$. Notice that $\widetilde{X_T}$ is a closed shift-invariant subspace of $X^\N$. The {\it natural extension} of $(X,T)$ is the dynamical system $(\widetilde{X_T}, \widetilde{T})$, where  $\widetilde{T}$ is the continuous maps on $\widetilde{X_T}$ which sends $(x_k)_{k\in \mathbb N}$ to $(Tx_0,x_0,x_1,\cdots )$. The map  $\widetilde{T}$ is an homeomorphism  of $\widetilde{X_T}$. Moreover the map $\pi:(x_k)_{k\in \N}\mapsto x_0$ semi-conjugates $\widetilde{T}$ and $T$, i.e. $\pi\circ \widetilde{T}=T\circ \pi $, and $\pi$ is surjective when $T$ is surjective. Moreover, if $T$ is a homeomorphism, then $(X,T)$ and $(\widetilde{X_T},\widetilde{T})$ are topologically conjugated. The invertible topological system $(\widetilde{X_T},\widetilde{T})$ satisfies the following universal property : for any extension $\psi:(Y,S)\rightarrow (X,T)$ by an invertible topological system $(Y,S)$ there is a unique extension 
	  $\widetilde{\psi}: (Y,S)\rightarrow (\widetilde{X_T}, \widetilde{T})$ satisfying $\psi=\pi\circ \widetilde{\psi}$, which is given by $\widetilde{\psi}(y)=\left(\psi(S^{-k}y)\right)_{k\in \N}$ for all $y\in Y$.

	\subsection{Dimension of $\widetilde{X_T}$}
		The inverse limits of topological spaces and  their  topological aspects have been highly investigated (see \cite{ingram} and the references therein). We are interested  in dimension properties of  $\widetilde{X_T}$. An inverse limit of compact metric spaces does not raise dimension  : with the above notations, the dimension of $\varprojlim\mathbf{f}$ is less than or equal to $\liminf_{i\to \infty}\dim (X_i)$  (Corollary 182 \cite{ingram} ). But in general the inverse limit may have lower dimension, even when one consider a single bonding  surjective map, i.e. the dimension of $\widetilde{X_T}$ may be less than the dimension of $X$ for a surjective dynamical topological system $(X,T)$.

		 We give two general constructions of surjective maps $T:X\circlearrowleft$ with    $\dim(\widetilde{X_T})<\dim X$. Moreover we may assume $X$ is connected for $\dim(\widetilde{X_T})>1$. The first one is inspired from Example 183 in \cite{ingram}, whereas the second construction seems to be new. Contrarily to the first one, the second one allows to build finite-to-one examples.  In particular we will show the following proposition.
		 
			\begin{prop}\label{examples}
		For any integers $0\leq k\leq n$, there exists a surjective finite-to-one topological system $(X,T)$ with  $ \dim \widetilde{X_T}=k$ and $\dim(X)=n$. For $k\geq 1 $ we may assume $X$ is connected.
\end{prop}

			We recall that the inverse limit of continua (i.e. compact connected space) is also a continuum. In particular when $\widetilde{X_T}$ is not reduced to a point, its topological dimension is positive. Therefore in the above proposition, the set of states  $X$ can not be connected for $k=0$. 
	
	\subsubsection{Factor of a lower dimensional space}

\begin{lem}\label{ingramgen}
Let  $(X,T)$ be a topological dynamical system. Assume there are  subsets $Y,Z$ of $X$, such that 
		\begin{itemize}
		\item $X=Y\cup Z$,
		\item $T(Z)=Y$ and $T(Y)=Z$,
		\item $Y$ and $Z$ are  closed subsets of $X$.
		\end{itemize}
		Then we have 
		$$\dim(\widetilde{X_T})\leq \min (\dim Y, \dim Z).$$
\end{lem}	
\begin{proof}
We define a sequence of topological spaces $X_i$ and maps $f_i: X_{i+1}\to X_i$ as follows.  Let $X_i=Y$ for $i$ even and $X_i=X$ for $i$ odd and take $f_i$ be the restriction of $T$ to $X_{i+1}$. Let $\varprojlim  \mathbf{f}$ be the induced inverse limit. We let ${\bf f'}$ be the family obtained by inverting the role of odd and even numbers. Then $\widetilde{X_T}$ is contained in the union of $\varprojlim  \mathbf{f}$  and $\varprojlim  \mathbf{f'}$. By \cite[Corcollary 182]{ingram}, the dimensions of $\varprojlim  \mathbf{f}$ and $\varprojlim  \mathbf{f'}$ are less than or equal to the dimensions of $Y$ and $Z$. This completes the proof.
\end{proof}	

In the settings of Lemma \ref{ingramgen} we have $\dim(X)=\max(\dim Y,\dim Z)$. Therefore to get an example with $\dim(\widetilde{X_T})<\dim X$ it is enough to take 
surjective continuous  maps $R:Y\rightarrow Z$ and $S:Z\rightarrow Y$ with compact metric spaces $Y,Z$ satisfying $\dim Y<\dim Z$ and consider $T$  on the disjoint union $X=Y\coprod Z$ with $T|_{Y}=R$ and $T|_{Z}=S$. It is not difficult to produce such examples satisfying Proposition \ref{examples}, but the space $X$ is not connected (because it is a disjoint union of non empty closed sets). This problem  may be avoided thanks to the following corollary.

\begin{cor}[Example 183 in \cite {ingram}]\label{ingram} Let $(X,T)$ be a topological system. Assume there is a compact metric space $Y$ and surjective maps $R:X\rightarrow Y$ and $S:Y\rightarrow X$ with $T=S\circ R$. Then we have $$\dim(\widetilde{X_T})\leq \dim Y.$$
\end{cor}
\begin{proof}
Consider the surjective map $T'$ on the disjoint union $X\coprod Y$ with $T'|_X=R$ and $T'|_Y=S$. Then $(X,T )$ is the restriction of $T'^2$ on $X$, therefore we get 
$\dim (\widetilde{X_T}) \leq \dim (\widetilde{X_{T'^2}})$. By the subsequence theorem \cite{ingram}, we have $\dim (\widetilde{X_{T'^2}})=\dim(\widetilde{X_{T'}})$, so that we finally get by Lemma \ref{ingramgen} that 
$\dim (\widetilde{X_T})\leq \dim(\widetilde{X_{T'}}) \leq \dim Y$.
\end{proof}

In  \cite[Example 183]{ingram} the authors consider more precisely $X=[0,1]^n$ and $Y=[0,1]\times \{0\}$ with $S$ being the projection on the first coordinate and $R$ be a space-filling curve, then $\dim \widetilde{X_T}=1<\dim X=n$. In general, we may generalize the above example to general $k\ge1$ by replacing $X$ and $Y$ by $[0,1]^{n+k-1}$ and $[0,1]^k\times \{0\}$ respectively. 


We remark that the examples  produced as above are never finite-to-one. Indeed, in the settings of Lemma \ref{ingramgen}  for example, we have $\dim Z\leq\dim Y+\sup_{z\in Z}\dim   T^{-1}z$ (see for example \cite[Theorem 1.12.4]{engelking1995theory}). Then if $T$ is finite-to-one, its  fibers  are zero-dimensional and  we have necessarily  $\dim Z\leq\dim Y$ and then $\dim Y=\dim Z $ by symmetry.

\subsubsection{Examples with a lower dimensional subsystem}\label{sec:Examples with a lower dimensional subsystem}

	\begin{lem}\label{prop: dimension of natural ext}
		Let  $(X,T)$ be a topological system. Assume there are subsets $Y,Z$ of $X$, such that 
		\begin{itemize}
		\item $X=Y\cup Z$,
		\item $T(Z)\subset Z$,
		\item $Y$ and $Z$ are  closed subsets of $X$.
		\end{itemize}
		Then we have  $$\dim(\widetilde{X_T})\leq \max\left(\dim Y, \dim\left( \bigcap_{n\in \mathbb N}T^nZ \right) \right).$$
		
	\end{lem}

\begin{proof}
For each $k\in \mathbb N$ we let $\widetilde{Y_k}$ be the subset of $\widetilde{X_T}$ consisting in $(x_l)_{l\in \mathbb N}$ with $x_l\in Y$ for $l\geq k$.     Then we have 
$$\widetilde{X_T}=\widetilde{Z_T}\cup \bigcup_{k\in \mathbb N}\widetilde{Y_k}.$$
The sets $\widetilde{Y_k}$, $k\in \mathbb N$, are compact subsets of $X$ and the inverse limit $\widetilde{Z_T}$ has topological dimension less than or equal to $\dim \left(\bigcap_{n\in \mathbb N}T^nZ\right)$. Moreover we have $\dim(\widetilde{Y_k})\leq \dim Y$ for any $k$. By the countable union theorem (e.g. Theorem 1.7.1 in \cite{coornaert2015topological}), we conclude that  \begin{align*}
\dim(\widetilde{X_f})&\leq \max(\dim Y, \dim \widetilde{Z_T}),\\
&\leq \max(\dim Y, \dim \bigcap_{n\in \mathbb N}T^nZ ).
\end{align*}
\end{proof}

We are now in a position to prove Proposition \ref{examples}. 
\begin{proof}[Proof of Proposition \ref{examples}]
	
		Let $n>k\in \mathbb N$. We choose $Y$ equal to the standard Cantor set $C\subset [0,1]$ and $Z$ be the unit Euclidean cube $[0,1]^n$. Write $Y$ as the union of $Y_1:=C\cap [0,1/3]$  and $Y_2=C\cap [2/3,1]$, which are both homeomorphic to $C$. There is a finite-to-one continuous surjective map $\phi:Y_2\simeq C \rightarrow Z$. For example, for $n=1$, we may take $\phi : \{0,1\}^{\mathbb N^*}\rightarrow [0,1]$, $(a_k)\mapsto \sum_{k\ge 1} \frac{a_k}{2^k}$.  Then we  let  $T:Y\rightarrow Y\coprod Z$ be equal to $x\mapsto 3x$ on $Y_1$ and $T=\phi$ on $C_2$. Then we have $T(Y)=Y\coprod Z$. On $Z$ we let $T(y_1,\cdots,y_n)=(y_1,\cdots, y_k,y_{k+1}/2,\cdots, y_{n}/2 )$, therefore $\bigcap_{n\ge 0}T^nZ=[0,1]^k\times \{0^{n-k}\}$. Let $X=Y\coprod Z$. By Proposition \ref{prop: dimension of natural ext}, we get $\dim (\widetilde{X_T})\leq k$. Finally as $T$ is the identity map on $[0,1]^k\times \{0^{n-k}\}$, this cube embeds in $\widetilde{X_T}$ and $\dim (\widetilde{X_T})\geq k$.
The space $X$ is not connected but for $k>0$ we may arrange the construction to ensure the connectedness of $X$. Take $Y=[-1,0]\times  \{0^{n-1}\}$ and let $T:[-1,-1/2]\times  \{0^{n-1}\}\rightarrow Z=[0,1]^n$ be a finite-to-one space filling curve (e.g. the Hilbert space filling curve) with $T(-1/2,0)=0^n$. On $[-1/2,0]\times \{0^{n-1}\}$ we let $T(x)=(F_2(x_1),0^{n-1})$,where $F_2$ denotes the continuous "tent"  map affine on $[-1/2,-1/4]$ and $[-1/4,0]$ with $F_2(-1/2)=F_2(0)=0$ and $F_2(-1/4)=-1$. By taking  $T$ on $Z$ as above, we get the desired example.  

\end{proof}
	


	\subsection{Mean dimension of $\widetilde{T}$}
	Let $(X,T)$ be a topological dynamical system and $(\widetilde{X_T}, \widetilde{T})$ be its natural extension.
	We study now the relation between the mean dimension of $\widetilde{T}$ and $T$. In this section, we show that the mean dimension of the natural extension is always less than or equal to the mean dimension of the system, but they may differ. 
	\subsubsection{Inequality}
	
Given a sequence $(X_i,T_i)_{i\in \mathbb N}$ of topological systems and a family of continuous maps $\mathbf{f}=(f_i:X_{i+1}\rightarrow X_{i})_{i\in \mathbb N}$, 	we may define the inverse limit system as the map $\varprojlim \mathbf{T}:\varprojlim \mathbf{f}\circlearrowleft$ which maps $(x_i)_{i\in \mathbb N}$ to $(T_ix_i)_{i\in \mathbb N}$. The second author proved $\mdim(\varprojlim \mathbf{f}, \varprojlim \mathbf{T})\leq \liminf_{i\to \infty}\mdim(X_i, T_i)$ in   \cite[Proposition 5.8]{shi2021marker} \footnote{Even though it is stated in \cite[Proposition 5.8]{shi2021marker} that $\mdim(\varprojlim \mathbf{f}, \varprojlim \mathbf{T})\leq \sup_{i\in \N}\mdim(X_i, T_i)$, it is indeed shown that $\mdim(\varprojlim \mathbf{f}, \varprojlim \mathbf{T})\leq \liminf_{i\to \infty}\mdim(X_i, T_i)$ by carefully checking its proof.}. 
For a topological system $(X,T)$, we consider $(X_i,T_i)=(X,T)$ and $f_i=T$ for all $i$. Then $(\varprojlim \mathbf{f}, \varprojlim \mathbf{T})$ is just the natural extension $(\widetilde{X_T},\widetilde T)$ of $(X,T)$. As a consequence we have in particular :

\begin{prop}\label{nat}
	Let $(X,T)$ be a topological dynamical system and $(\widetilde{X_T}, \widetilde{T})$ be its natural extension. Then we have 
\begin{equation}\label{eq:nat}
\mdim(X,T)\geq \mdim(\widetilde{X_T}, \widetilde{T}).
\end{equation}
\end{prop}	


		
We remark that the equality of \eqref{eq:nat} can be achieved: for example, the mean dimension of a unilateral full-shift is equal to the mean dimension of its natural extension, which is the corresponding bilateral shift. 

\begin{ques}
Does the mean dimension of a general CA coincide with the mean dimension of its natural extension?
\end{ques}

In some cases we answer positively to the above question in the next sections. 
		
	\subsubsection{A counterexample  $(X_C, T_C)$ to the equality of \eqref{eq:nat}}
	
The inequality in Proposition  \ref{nat} may be strict. We present below an example. 

Let $f:[0,1]\rightarrow \mathbb R$ be the $\times 3$-map., i.e. $x\mapsto 3x$ mod $1$. For each $n\in \mathbb N$ we let $C_n$ be the $n^{th}$ standard Cantor 
set, i.e. $C_n:=\bigcap_{0\leq l\leq n}f^{-l}([0,1/3]\cup [2/3,1])$. Observe that $f_n=f|_{C_n}:C_{n+1}\rightarrow C_n$ is surjective and $\bigcap_{n\in \N} C_n$ is the standard Cantor set $C$.  We consider the compact metrizable space $X_C=\prod_{n\in \mathbb N}C_n$ and the surjective map $T_C:X_C\circlearrowleft$ defined by 
$$\forall x=(x_n)_{n\in \N}\in X_C, \ T_C(x)=(3x_{n+1} )_{n\in \N}.$$

\begin{prop}\label{prop:example natural}
	We have
$$\mdim(X_C, T_C)\geq 1>\mdim(\widetilde{X_{T_C}},\widetilde{T_C})=0$$
\end{prop}

\begin{proof}
1. We first show $\mdim(X_C, T_C)\geq 1$. For each $n\in \N$, we let $I_n$ be the connected component of $C_n$ containing $0$. Then $f|_{I_{n+1}}:I_{n+1}\rightarrow I_n$ is an homeomorphism. The product $Y=\prod_{n\in \N}I_n\subset X_C$ satisfies $T_C(Y)\subset Y$ and the induced subsystem $(Y, (T_C)_{|Y})$ of $(X_C, T_C)$ is topologically conjugated to the full shift $([0,1]^{\mathbb Z }, \sigma)$ via the conjugacy 
\begin{align*}
\phi:&Y\rightarrow [0,1]^{\mathbb N }, \\
& (x_n)_{n\in \N}\mapsto (3^nx_n)_{n\in \N}.
\end{align*}
 Therefore $\mdim(X_C, T_C)\geq \mdim([0,1]^{\mathbb N }, \sigma)=1$.
\vspace{0,3cm}

2. We check now that $\mdim(\widetilde{X_{T_C}},\widetilde{T_C})=0$. In fact we will show that $(\widetilde{X_{T_C}},\widetilde{T_C})$ is topologically conjugated to $(C^{\mathbb Z}, \sigma)$. An element $x$ of $\widetilde{X_{T_C}}$ may be written under the form 
$x=(x_n^k)_{\stackrel{n\in \mathbb N}{k\in \mathbb N}}$ with $x^k=(x_n^k)_{n\in \N} \in X_C=\prod_{n\in \N}C_n$ and $x=(x^k)_{k\in \N}\in \widetilde{X_{T_C}}$.
Moreover the Cantor set is the inverse limit of the family $\mathbf{f}=( f_n)_{n\in \mathbb N}$  (recall  $f_n: C_{n+1}\rightarrow C_{n}$  is the  the $\times 3$-map). We will  use this identification $C=\varprojlim \mathbf f$. 
For $x\in \widetilde{X_{T_C}}$ we let $x^k= T_C^{-k}(x^0)$ for $k<0$ so that we have $T_C(x^{k+1})=x^{k}$ for all $k\in \Z$, i.e.  $3x^{k+1}_{n+1}=x^{k}_n$ for all $k\in\Z$ and $n\in \N$.

We consider the map $\phi:\widetilde{X_{T_C}}\rightarrow C^{\mathbb Z}$ defined by 
\begin{align*}\forall x=(x_n^k)_{k,n}\in \widetilde{X_{T_C}}, ~& \phi(x)=(y^k)_{k\in \Z}\\
\text{ with } y^k &=(x_{n}^{n-k})_{n\in \mathbb N}\in C.
\end{align*}
This map takes value in $ C^{\mathbb Z}$ because $3y^{k}_n=3x_{n}^{n-k}=x_{n-1}^{n-k-1}=y^k_{n-1}$ for all $k\in\mathbb Z$ and $n\in \N$.
Clearly $\phi$ is continuous and bijective with inverse $\phi^{-1}=\phi$. 
Finally  we check easily that
\begin{align*}
 \phi \circ \widetilde{T_C}(x)&=\phi( (x^{k-1})_{k\in \N} ),\\
 &= (y^{k+1})_{k\in \Z},\\
 &= \sigma \circ \phi(x).
\end{align*}
This completes the proof.
	

\end{proof}

\subsubsection{Metric mean dimension of $(X_C, T_C)$}
For a compact metric space $(X,d)$ and a sequence $\delta=(\delta_i)_{i\in \N}$ of positive numbers going to zero, we let $d_{\delta}$ be the distance on $X^{\mathbb N}$ defined by 
$$\forall x,y\in X^{\N}, \ d_{\delta}(x,y)=\sup_{i\in \N}\delta_i d(x_i,y_i).$$ A metric equivalent to some $d_\delta$ will be called a product metric and   will be denoted by $d^\N$. Such a product metric is compatible with the product topology on $X^\N$.

We show in the Appendix \ref{app:A} that for any zero-dimensional system $(X,T)$, there is a metric on $D$ on $X$ with $\mdim_M(X,T,D)=\mdim(X,T)=0$. In particular there is such a metric $D_C$ for the system $(\widetilde{X_{T_C}},\widetilde{T_C})$. 
One may wonder  if $D_C$ is a product metric. The following lemma applied to $(X_C, T_C)$ shows that it is not the  case.

\begin{lem}
Let $(X,d,T)$ be a surjective topological dynamical system. Then for any product metric $d^\N$ 
$$\mdim_M(\widetilde{X_T},\widetilde{T},d^{\mathbb N})=\mdim_M(X,T,d).$$
\end{lem}

\begin{proof}
	One only needs to consider the case of $d^{\N}=d_\delta$ for some  sequence   $\delta=(\delta_i)_{i\in \mathbb N}$ with $\lim_{i\rightarrow \infty}\delta_i=0$ and $0<\delta_i\leq 1$ for all $i$, because any product metric is equivalent to a metric of this form.
	\vspace{0,3cm}

1. \underline{$\mdim_M(\widetilde{X_T},\widetilde{T},d^{\mathbb N})\geq \mdim_M(X,T,d)$.} Let $\epsilon>0$ and $n\in \mathbb N$. Let $E$ be a $(n,\epsilon)$-separated set of $(X,d,T)$. By surjectivity of $T$, there is for each $x\in E$ a point $\tilde{x}=(x_k)_{k\in \N}\in \widetilde{X_f}\subset X^\N$ with $x_0=x$. Let $\widetilde{E}=\{\tilde{x}, \ x\in E\}$. Then for all $x\neq y\in E$, there is $0\leq k<n$ with  $d(f^kx,f^ky)\geq \epsilon$ so that we have 
\begin{align*}
d_\delta(\widetilde{T}^k\tilde{x},\widetilde{T}^k\tilde{y})&\geq \delta_0 d(f^kx_0,f^ky_0),\\
&\geq  \delta_0d(f^kx,f^ky)\geq \delta_0\epsilon.
\end{align*}
Therefore $\widetilde{E}$ is a $(n,\delta_0\epsilon)$-separated subset of $(\widetilde{X_T},d_{\delta},\widetilde{T})$. One easily concludes that $\mdim_M(\widetilde{X_T},\widetilde{T},d^{\mathbb N})\geq \mdim_M(X,T,d)$.
\vspace{0,3cm}

2. \underline{$\mdim_M(\widetilde{X_T},\widetilde{T},d^{\mathbb N})\leq \mdim_M(X,T,d)$.} Fix $\epsilon>0$. Let $N>0$, such that $\sup_{i\geq N}\delta_i \text{diam}_d(X)<\epsilon$. For $n\in \N$, we let $Y$ be a  $(\epsilon,n)$-separated set of $(\widetilde{X_T},d_\delta,\widetilde{T})$.  Denote by $\pi:(\widetilde{X_T},\widetilde{T})\rightarrow (X,T)$, $(x_k)_{k\in \N }\mapsto x_0$  the natural extension.  For $x\in Y$ we let $E_x$ be the set of points $y\in Y$  such that $\pi(x)$ and $\pi(y)$ are not $(\epsilon,n)$-separated with respect to $d$.  Observe that for any $y\in E_x$, there is  $i,j\in [0,N[$ with $(d(x_i,y_i)\geq) \,  \delta_jd(x_i, y_i)\geq \epsilon$ (if not we would have $d_\delta(\widetilde{T}^kx,\widetilde{T}^ky)<\epsilon$ for all $0\leq k<n$).


 Therefore the cardinality of $E_x$ is bounded from above by some constant $C$ depending only on $N$ and $\epsilon$ (namely  the maximal cardinality $C=C(\epsilon, N)$ of $\epsilon$-separated set in $X^N$ for  the usual finite product distance $d^N$ on $X^N$). 
 Consequently there is a  $(\epsilon,n)$-separated  subset $Z\subset\pi(Y)$  of $(X,d,T)$ with  $C\sharp Z\geq   \sharp Y$. For  example we may take   any $z_0=\pi(x_0)$ in $\pi(Y)$, then $z_1=\pi(x_1)$ in $\pi(Y\setminus E_{x_0})$,   $z_2=\pi(x_2)$ in $\pi\left(Y\setminus (E_{x_0}\cup E_{x_1})\right)$, etc. The process stop at some $N$ with $CN\geq \sharp Y  $ and we may then let $Z=\{z_1,z_2,\cdots z_N\}$.
Consequently  $h_{d^{\mathbb Z}}(\widetilde{T},\epsilon)\leq h_d(T,\epsilon)$ for all $\epsilon>0$.   
\end{proof}

\subsection{Mean dimension of a skew-product}	

		Let $X$ and $Y$ be compact spaces. Let $R: X\to X$ and $S: X\times Y\to Y$ be  continuous maps. Define  the {\it skew-product} $T: X\times Y\circlearrowleft$ over $(X,R)$ by $(x,y)\mapsto (R(x), S(x,y))$. 
		
	\begin{lem}\label{lem:skew product}	Let $\alpha$ be an open cover of $X$ and $\beta=\alpha\times Y$ be the induced cover of $X\times Y$. Then 
	$$\mdim(X\times Y, T,\beta)\geq \mdim(X, R, \alpha). $$
	In particular, we have 	$\mdim(X\times Y, T)\ge \mdim(X,R)$.
	\end{lem}
	\begin{proof}
	 For any $n\in \N$, we let $\gamma_n$ be a finite open cover of $X\times Y$ finer  $\bigvee_{k=0}^{n-1} T^{-k}\beta $ with $D(\bigvee_{k=0}^{n-1} T^{-k}\beta )=   \text{ord} (\gamma_n)$. Fix $y\in Y$ and consider  the  open cover $\gamma'_n$ of $X$ given by the sets $\pi_X\left (O\cap (X\times\{y\})\right)$ over $O\in \gamma_n$, where $\pi_X:X\times Y\rightarrow X$ denotes the projection on the $X$-coordinate. Clearly $\gamma'_n$ is finer then 
 $\bigvee_{k=0}^{n-1} R^{-k}\alpha$ and $\text{ord}(\gamma'_n)\leq \text{ord}(\gamma_n)$. Therefore $\frac{D(\bigvee_{k=0}^{n-1} T^{-k}\beta )}{n}\geq \frac{ D(\bigvee_{k=0}^{n-1} R^{-k}\alpha )}{n}$ and we conclude by taking the limit in $n$. 
 \end{proof}
	

\begin{rem}
The above proof still applies in the wider context of a topological  system $T:E\circlearrowleft$ with $E\subset X\times Y$,  $\pi_X(E)=X$  (where $\pi_X$ denotes the coordinate projection on $X$) and $R\circ \pi_X=\pi_X\circ T$, such that there exists $y$ with $X\times \{y\} \subset E$. 
\end{rem}	
	
\begin{rem}	\label{rem:nat skew product}
In general the natural extension of a skew-product is not a skew-product. When the skew-product is trivial, i.e. with $S$ depending only on $x$, then the natural extension $(\widetilde{X_R}, \widetilde{R})$ and $(\widetilde{(X\times Y)_T},\widetilde{T})$ are topologically conjugated via the map
$(x_k)_k\in \widetilde{X_R} \mapsto (x_k, S(x_{k+1}))_k\in \widetilde{(X\times Y)_T}$, in particular these systems have the same mean dimension.
 \end{rem}
	
	\section{General one-dimensional cellular automata}
		Let $X$ be a compact metrizable space.  Let $F$ be a cellular automaton  on $X^{\mathbb Z}$ with a continuous  local rule $f:X^I\rightarrow X$ for $I \subset \Z$. For $x=(x_k)_{k\in \Z}\in X^{\mathbb Z}$ and a finite subset $K$ of $\mathbb Z$  we denote by $x_K$ the tuple given by the $j$-coordinate of $x$ for $j\in K$, i.e.
	$$
	x_{K}=(x_{j_1}, x_{j_2}, \dots, x_{j_k})~\text{for}~K=\{j_1<j_2<\cdots<j_k\}.
	$$
	Let $J=J_F$ be the integers in the convex hull of $I\cup \{0\}$ and   $J^*=J\setminus \{0\}$. Let  $J_-=\min\{J\}$ and $J_+=\max\{J\}$.   Then $\text{\rm diam}(I\cup\{0\})=\text{\rm diam}(J)=\sharp J^*$. We denote again by $f$ the function from  $X^J$ to $X$,  mapping $(x_j)_{j\in J}$ to $ f\left((x_j)_{j\in I}\right)$.

\subsection{Upper bound for the mean dimension }
	We first generalize the upper bound of the mean dimension w.r.t. the shift map obtained in \cite[Proposition 3.1]{LindenstraussWeiss2000MeanTopologicalDimension} to cellular automata.

	\begin{prop}\label{prop:upper bound}
		Let $X$ be a compact metric space.  Let $F$ be a cellular automaton  on $X^{\mathbb Z}$ with a continuous  local rule $f:X^I\rightarrow X$ for $I \subset \Z$. Then $\mdim(X^{\Z}, F)\le \stab(X)\cdot {\rm diam}(I\cup \{0\}).$
	\end{prop}
	
	\begin{proof}
		For every $K\subset \Z$, let $\pi_K: X^\Z \to X^K$ be the natural projection. Let $J$ be the integers in the convex hull of $I\cup \{0\}$. Let $\AA$ be a finite open cover of $X^\Z$. There is a finite open cover $\mathcal B$ finer than $\AA$ such that for some positive integer $N$  $$\mathcal B \subset \mathcal{O}(-N,N):=\{ \{(x_n)_{n\in \Z}, \ x_{-N}\cdots x_N\in O\} \ : \  O~\text{is open in}~X^{2N+1}\}.$$ By assumption of $F$, we get that $$\bigvee_{k=0}^{n-1}F^{-k}\BB \subset \mathcal{O}(-N+nJ_- , N+nJ_+ ).$$ Let $K_n$ be the integers in $[-N+nJ_- , N+nJ_+]$.
		It follows that there is a cover $\mathcal C\succ \pi_{K_n}  (\bigvee_{k=0}^{n-1}F^{-k}\BB)$ of $X^{K_n}$ such that ${\rm ord}(\mathcal C)\le \dim (X^{ K_n})$. As $\pi_{K_n}^{-1} \mathcal C\succ  \bigvee_{k=0}^{n-1}F^{-k}\BB\succ  \bigvee_{k=0}^{n-1}F^{-k}\AA$, we obtain that
		$$
		\frac{D(\bigvee_{k=0}^{n-1}F^{-k}\AA)}{n}\le \frac{\dim (X^{ K_n})}{n} = \frac{\dim (X^{ K_n})}{\sharp K_n} \cdot \frac{2N+n\cdot {\rm diam}(J)}{n}. 
		$$
		Therefore, we conclude that $\mdim(X^{\Z}, F)\le \stab(X)\cdot {\rm diam}(J).$
	\end{proof}
	
A cellular automaton  on $X^{\mathbb Z}$ with  local rule $f:X^I\rightarrow X$  is  also a cellular automaton with local rule $f':X^K\rightarrow X$ for $K\supset I$ by letting  $f'(x_K)=f(x_I)$. Thus we need some extra conditions on local rules in order to calculate the value of mean dimension.


	\subsection{A factor of CA}\label{previ}

	Define $\phi:X^{\mathbb Z}\rightarrow X\times \left(X^{J^*}\right)^{\mathbb N}$ by $$x=(x_n)_{n \in \mathbb Z}\mapsto (x_0,(x_{J^*}, F(x)_{J^*}, \cdots , F^{k}(x)_{J^*},\cdots)).$$ 
	and  $g:X\times \left(X^{J^*}\right)^{\mathbb N}\circlearrowleft $ by
	$$\left(x,(y^n)_{n\in \mathbb{N}} \right)\mapsto (f(z), \sigma y)$$
	where $z$ is the  point of  $X^J$ defined as  $z_0=x, (z_i)_{ i\in J^*}=y^0$. 
	For any $ x\in X^{\mathbb Z}$, we have then
	\begin{align*}
	\phi\circ F(x)&=(f(x_J),(F(x)_{J^*}, F^2(x)_{J^*}, \cdots , F^{k}(x)_{J^*},\cdots) =g\circ\phi(x).
	\end{align*}
	In general $\phi$ is not surjective, but when this is the case, the continuous map $\phi$ defines a factor map from $(X^{\mathbb Z}, F)$ to $(X\times X^{\mathbb N},g)$.
	
	\begin{cor}\label{cor:mdim X times X N}
	 $$\mdim\left(X\times \left(X^{J^*}\right)^{\mathbb N}, g\right)\ge \stab(X) \cdot \text{\rm diam}(I\cup\{0\}).$$
	\end{cor}
	\begin{proof}
		By Lemma \ref{lem:skew product} and \cite[Theorem 1.1]{tsukamoto2019mean}, we have
		$$\mdim\left(X\times \left(X^{J^*}\right)^{\mathbb N}, g\right)\ge \mdim\left( \left(X^{J^*}\right)^{\mathbb N}, \sigma\right)=\sharp J^* \cdot \stab(X).$$
	\end{proof}
			
For higher dimensional CA, we may generalize the above semi-conjugacy as follows. Let $J$ be a subset of $\mathbb Z^d$. Then for any $J'$ contained in $J$ satisfying $k+I\subset J$ for all $k\in J'$, we define $\phi=\phi_{J, J'}$ by $\phi:X^{\mathbb{Z}^d}\rightarrow X^{J'}\times \left(X^{J\setminus J'}\right)^{\mathbb N}$ by $$x=(x_n)_{n \in \mathbb Z}\mapsto (x_{J'},(x_{J\setminus J'},F(x)_{J\setminus J'}, \cdots , F^{k}(x)_{J\setminus J'},\cdots)).$$
 and $g=g_{J,J'}:X^{J'}\times \left(X^{J\setminus J'}\right)^{\mathbb N}\circlearrowleft $ by
	$$\left(x,(y^n)_{n\in \mathbb{N}} \right)\mapsto \left( \left(f(z_{k+I}) \right)_{k\in  J'}, \sigma y\right)$$
	where
 $z_{k+I}:=(z_{k+i})_{i\in I}$ is the  point of  $X^I$ defined as  $z_{k+i}=x_{k+i}$ for  $k+i\in  J'$ and  $z_{k+i}=y^0_{k+i}$ for $k+i\in J\setminus J'$.
 Then we have again  $\phi\circ F(x)=g\circ\phi(x)$ for all $x\in X^{\mathbb{Z}^d }$.

	\subsection{CA with surjective local rule}\label{f}

	For a topological dynamical system $(X,T)$ we let $NW(T)$ be the set of non-wandering points of $(X,T)$. Let $F$ be a cellular automaton with local rule $f:X^I\rightarrow X$. For  positive integers $n$ we define by induction  the subsets $X_n$ of $X$ by $X_n= f(X_{n-1}^I)$ and $X_0=X$. Finally we let $X_\infty:=\bigcap_{n\in \mathbb N}X_n$, which is a compact subspace of $X$. Clearly $f(X_\infty^I)=X_\infty$. Therefore the restriction of  $F$ to $X_\infty^{\mathbb Z}$ is a CA with a surjective local rule.

	\begin{prop}\label{ima}
		With the above notations, 
		$$\mdim(X^\Z, F)=\mdim(X_\infty^{\mathbb Z}, F|_{X_\infty^{\mathbb Z}}).$$
	\end{prop} 
	\begin{proof}
		Since $X_\infty \subset X$, $\mdim(X^\Z, F)\ge \mdim(X_\infty^{\mathbb Z}, F|_{X_\infty^{\mathbb Z}})$. It remains to show the other direction.
		Notice that for a general topological system $(X,T)$ we always have  $NW(T)\subset \bigcap_{n\in \mathbb N}T^n X$. Moreover  $\mdim(X, T)=\mdim(NW(T), T|_{NW(T)})$ by \cite[Lemma 7.2]{gut17}.
		But $F^n(X^\mathbb Z)\subset X_n^{\mathbb Z}$ for all $n$, therefore $NW(F)\subset X_\infty^{\mathbb Z}$, implying that $\mdim(X^\Z, F)\le \mdim(X_\infty^{\mathbb Z}, F|_{X_\infty^{\mathbb Z}})$.
	\end{proof}
	
	When considering the mean dimension, by Proposition \ref{ima} and the argument above, we could restrict to CA's with surjective local rule.

	\section{Permutative one-dimensional CA}\label{sec:Permutative CA}
	Let $X$ be a compact metric space.  Let $f$ be a continuous function  $f:X^I\rightarrow X$ with  $I \subset \Z$.  For any $j\in I$ and  for any $x^j=(x_i)_{i\in I\setminus \{j\}} \in X^{I\setminus \{j\}}$,  we denote by  $f_{x^j}:X\circlearrowleft $ the continuous function $x_j \mapsto f((x_i)_{i\in I})$.

	A cellular automaton $F$ on $X^{\mathbb Z}$ with   local rule $f$ is said to be {\it permutative}
	(resp. {\it strongly permutative}) when for $j\in \{\max I,\min I\}\setminus \{0\}$  and for all $x^j
	\in X^{I\setminus \{j\}}$ the map $f_{x^j}$ is surjective (resp. bijective \footnote{By 
		compactness of $X$, the map $f_{x^j}$ is then an homeomorphism.}).
	
	If $F$ is permutative (resp. stongly permutative) then so is $F^k$  and $J_{F^k}=kJ_F$ for any $k\in \mathbb N^*$ (see Lemma 16 in \cite{Bur20}). Note that in the discrete case, any  permutative CA is  strongly permutative as the set of states is finite. 
	
	\subsection{Maximal mean dimension for strongly permutative CA's}	
	
	When  there are no negative integers (resp. positive) in the domain $I$,  the local rule $f$ induces a cellular automaton on $X^{\mathbb N}$ (resp. $X^{-\mathbb N}$) which we denote respectively by $F^+$ and $F^-$. We let
	\begin{equation*}
	(\mathbb Y,G)=
	\begin{cases}
	(X^\mathbb{Z},F), ~&\text{if $I$ contains both  negative and positive integers},\\
	(X^\mathbb{N},F^+)~&\text{if $I$ contains no  negative integers},\\
	(X^{-\mathbb{N}},F^-)~&\text{if $I$ contains no positive integers}.
	\end{cases}
	\end{equation*}

	\begin{lem}\label{lem:topological extension}Let $F$ be a  permutative  cellular automaton. 
		The dynamical system $(\mathbb Y, G)$ is a topological extension of  $(X\times X^\N, g)$ via $\phi$.
	\end{lem}
	\begin{proof}
By Subsection \ref{previ} we only need to show the surjectivity of $\phi$. Let $(x, (y^k)_{k\in\mathbb N})\in X\times(X^{J^*})^{\mathbb N}$. 
Let us  show that there is $z\in X^\mathbb{Z}$ with $\phi(z)=(x, (y^k)_{k\in\mathbb N})$. We prove by induction on $k$ that there exist $z
\in \mathbb{Y}$ such that we have  $z_0=x$, $(z_j)_{j\in J^*}=y^1$, ..., $F^k(z)_{J^*}=y^k$. Assume it holds for $k$ and $1\in J=J_F$.
   Recall that $F^{k+1}$ is permutative and $J_{F^{k+1}}=(k+1)J_F$. Therefore we may change only the $K$-coordinate $z_K$ of $z$ with  $K=(k+1)J_++1$ to ensure $(F^{k
+1}z)_1=y^{k+1}_1$. Then we argue similarly for the other $j$-coordinates of $F^{k+1}z$ increasingly in $j\in J^*\cap \mathbb N$. Finally 
we may consider in the same way negative $j\in J^*$ starting from $-1$ and going decreasingly to $j=J_-$.    
	\end{proof}

	\begin{lem}\label{lem:conjugate}
	 Let $f:X^I\rightarrow X$ be a  strongly permutative  local rule.  Then 
	 $(\mathbb Y, G)$ is  topological conjugated to $(X\times X^\N, g)$ via $\phi$.	 
	\end{lem}
	
	\begin{proof}
	It is enough to notice that for a strongly permutative CA, the sequence $y\in \mathbb{Y}$ in the proof of Lemma \ref{lem:topological extension} is uniquely defined, so that $\phi:\mathbb Y\rightarrow X\times X^\N$ is a homeomorphism.
	\end{proof}

	We obtain the following formula concerning about the mean dimension of strongly permutative cellular automata.
	\begin{thm}\label{strong}
		Let $F$ be a strongly permutative cellular automaton as above, then 
		$\mdim(X^{\Z}, F)=\stab(X) \cdot \text{\rm diam}(I\cup\{0\}).$
	\end{thm}
	\begin{proof}
		When $I$ contains both negative and positive integers, the result follows from Proposition \ref{prop:upper bound}, Lemma \ref{lem:conjugate} and Corollary \ref{cor:mdim X times X N}. Now assume $I\subset \mathbb N$ (one deals similarly for the remaining case). Then $(\mathbb Y, G)=(X^{\mathbb N},F^+)$. Notice that $(X^{\mathbb Z}, F)$ is a skew-product over $(X^{\mathbb N},F^+)$, so that by Lemma \ref{lem:skew product}  		we have $\mdim(X^{\mathbb Z}, F)\geq \mdim(X^{\mathbb N},F^+)=\mdim(X\times X^\N, g)$. Then we conclude as in the previous case by using  Proposition \ref{prop:upper bound} and Corollary \ref{cor:mdim X times X N}. 
	\end{proof}

	\begin{rem}
	When the local rule does not depend on the zero coordinate, i.e. $0\notin I$, then 
	by Remark \ref{rem:nat skew product} the natural extension of the strongly permutative CA $F$ is topologically conjugated to the bilateral shift on $X^{J^*}$. In particular $F$ and its natural extension have the same mean dimension. 
	\end{rem}

	We will show in the next  section that  Theorem \ref{strong} does not holds true anymore for general permutative CA  by building  examples of permutative CA with intermediate mean dimension, i.e. with mean dimension strictly less than $\stab(X)\cdot\text{\rm diam}(I\cup\{0\})$. For permutative CA, we have $X_\infty=X$, therefore  Proposition \ref{ima} is useless to produce such examples.

		\subsection{Maximal mean dimension for near strongly permutative CA}	
	Let $\mathcal C(X)$ be the set of continuous maps from  a compact metrizable space $X$ to itself endowed with the topology of uniform convergence. A family $\mathcal F$ in $\mathcal C(X)$ is said to be \textit{$m$-expansive} if one of the equivalent  conditions is satisfied:
\begin{itemize}
\item there exists an open cover $\alpha$ of $X$ with $\mdim(X,T, \alpha)=\mdim(X, T)$ for every $T\in \FF$. Such a cover $\alpha$ is called a {\it generator} of $\FF$, 
\item there exist $\epsilon>0$ and a compatible metric $d$ with $\mdim(X, T,  d, \epsilon)=\mdim(X, T)$ for every $T\in \FF$, 
\item for all compatible metric $d$ there exists $\epsilon>0$ with $\mdim(X, T,  d, \epsilon)=\mdim(X, T)$ for every $T\in \FF$.
\end{itemize}

\begin{lem}\label{lem:m-expansive}
Let $\mathcal F$ be a $m$-expansive family. Then $T\mapsto \mdim(X, T)$ is upper semi-continuous on the closure of $\mathcal F$. 
\end{lem}

\begin{proof}
	Let $\alpha$ be a generator of $\FF$.
We will prove $T\mapsto \mdim(X, T,\alpha)$ is upper semi-continuous. The conclusion then follows : if $T_n\xrightarrow{n}T$ with $T_n\in \mathcal F$ for all $n$, then we get 
\begin{align*}
\mdim(X, T)&\geq \mdim(X, T,\alpha),\\
&\geq \limsup_{n\to \infty} \mdim(X, T_n,\alpha),\\
&\geq  \limsup_{n\to \infty} \mdim(X, T_n).
\end{align*}
To check the upper semi-continuity of $T\mapsto \mdim(X, T,\alpha)$, it is enough to see $f_n : T\mapsto D(\bigvee_{k=0}^{n-1}T^{-k}\alpha )$ is upper semi-continuous for any $n$. Indeed  we have $\mdim(X, T)=\inf_{n\in \N}\frac{f_n(T)}{n}$ by sub-additivity of the sequence $(f_n(T))_{n\in \N}$ and the infimum of a sequence of upper semicontinuous functions is itself upper semicontinuous. By  \cite[Proposition 1.6.5]{coornaert2015topological}, we have  $D(\bigvee_{k=0}^{n-1}T^{-k}\alpha )=\min\limits_{\gamma} \text{ord}(\gamma)$, where the minimum holds   over all closed covers $\gamma$ finer than $\bigvee_{k=0}^{n-1}T^{-k}\alpha$. Let $\gamma_n$ be such a cover realizing the minimum. It follows that $\gamma_n$ is finer than $\bigvee_{k=0}^{n-1}T_m^{-k}\alpha$ for $m$ large enough, so that $D(\bigvee_{k=0}^{n-1}T_m^{-k}\alpha )\leq 
\text{ord}(\gamma_n)=D(\bigvee_{k=0}^{n-1}T^{-k}\alpha )$ for $m$ large enough. It implies that $f_n$ is upper  semicontinuous function for every $n\in \N$.
\end{proof}

When a $m$-expansive family consists of a single map $T$, we say that the map $T$ is {\it $m$-expansive}.

\begin{lem}[\cite{tsukamoto2019mean}, Lemma 3.1 and Theorem 2.5]\label{tsu}
The shift map $\sigma$ on $X^\mathbb N$ is $m$-expansive. Moreover there is a generator $\alpha$ of the form $\alpha=\mathcal U\times X^{\mathbb{N}^*}$ for some open cover $\mathcal U$ of $X$.
\end{lem}

\begin{rem}This concept of $m$-expansiveness is inspired by the notion of $h$-expansiveness relative to the entropy. We may build examples of non $m$-expansive systems with any  given mean dimension as\ follows.  Consider  a sequence $(X_n, T_n)_{n\in \N}$ of topological systems such that $\mdim(X_n,T_n)$ is strictly increasing (either bounded or divergent to infinity). Then the one point  compactification $(X,T)$ of the disjoint union of $(X_n,T_n)$ by a $T$-fixed point at the infinity  is not $m$-expansive. 
\end{rem}

 For $N\in \mathbb{N}$ we denote by $\mathcal F_N(X)$ (resp. $\mathcal G_N(X)$ ) the family of strongly permutative CA  (resp. all CA) with domain $I$ contained in $[-N,N]$.

\begin{lem}\label{mexpex}
With the above notations, the family $\mathcal F_N(X)$ is $m$-expansive.
\end{lem}
\begin{proof}Let $J$ be an interval of integers contained in $[-N,N]$.
By Lemma \ref{tsu} there is a generator $\alpha=\mathcal U\times (X^{J^*})^{\mathbb{N}^*}$ of the shift map on $(X^{J^*})^{\mathbb N}$ for some cover $\mathcal U$ of $X^{J^*}$. Let $\beta =\mathcal U\times X^{\mathbb Z^*}$ and $\gamma=X\times \alpha$ be the induced covers of $X^{\mathbb Z}$ and $X\times (X^{J^*})^\mathbb{N}$ respectively. Let $F$ be a strongly permutative CA with $J$ being the convex hull of its domain $I$. We claim that $\beta=\beta_J$ is a generator of $F$. Indeed the conjugacy $\phi$ sends $\beta$ to $\gamma$, therefore $\mdim(F,\beta)=\mdim(g,\gamma)$. But $g$ being a skew-product overs $\sigma$, we have by Lemma \ref{lem:skew product} 
$\mdim(g,\gamma)\geq \mdim(\sigma, \alpha)=\mdim(F)$. 
Finally we observe that $\bigvee_{J}\beta_J$, where the joining holds over all intervals  of integers   $J\subset [-N,N]$,  is a generator of $\mathcal F_N(X)$.

\end{proof}

\begin{ques}
Is the larger family $\mathcal G_N(X)$ also $m$-expansive? 
\end{ques}

A CA on $X^\mathbb Z$ is said {\it near strongly permutative} when it belongs to the closure of $\mathcal F_N(X)$ for some $N$.  For example if $g$ and $h$ are surjective  non-injective monotone interval maps, then the CA with local rule $f:[0,1]^2\rightarrow [0,1]$, defined by $f(x_0,x_1)=g(x_0)+h(x_1)$, is near strongly permutative but not
strongly permutative.

\begin{cor}\label{near}
Let $F$ be a near strongly permutative CA. Then $$\mdim(F)=\stab(X)\cdot \text{\rm diam}(I\cup\{0\}).$$
\end{cor}
\begin{proof}
	It follows from Lemma \ref{mexpex}, Lemma \ref{lem:m-expansive} and Theorem \ref{strong}.
\end{proof}

	\subsection{Infinite topological entropy}
	
	\begin{lem}\label{entropy}		Assume $X$ is infinite. Then any permutative CA on $X^{\mathbb Z}$ has infinite topological entropy. 	\end{lem}	
		\begin{proof}		The factor map $\phi$ of Subsection \ref{previ} is surjective, therefore $h_{top}(F)\geq h_{top}(g)$. But $g$ is an extension of the unilateral full shift on $\left(X^{J^*}\right)^{\mathbb N}$ which has infinite entropy, therefore $h_{top}(F)=+\infty$. 	\end{proof}

However a cellular automaton  with infinite set of states and surjective local rule may have finite, even zero,  topological entropy. 
\begin{ex}
	Let $X=\{0,1,\frac{1}{2},\frac{1}{3}, \cdots \}$. Let $I=\{-1,0,1\}$. Define a surjective continuous map $f: X^I \to X$ by $f(x_{-1,}, x_0, x_1)=\max\{x_{-1,}, x_0, x_1 \}$. Notice that $F^n(x)\xrightarrow{n} a^{\Z}$ with $a=\max_{k\in \Z}x_k$ for every $x\in X^{\Z}$. Therefore  the topological entropy of $(X^\Z, T_f)$ is zero.
\end{ex}

	
	
\section{Unit CA }\label{sec:Unit CA}

In this section, we consider one-dimensional  CA  with domain $I=\{1\}$ and continuous rule $f:X\circlearrowleft$, i.e. 
	\begin{align}\label{form}
	F((x_k)_{k\in \Z})=(f(x_{k+1}))_{k\in \Z}.
	\end{align}
	We  call such CA a {\it unit CA}.

\subsection{Non-wandering set of unit CA's}	
	
	For unit CA's one may wonder if $NW(F)\subset NW(f)^{\mathbb Z}$. This is false in general. For example if one considers a North-South invertible dynamic on the circle $\mathbb S^1$, $NW(f)$ is reduced to the two poles and therefore  $h_{top}(F)\leq \log 2$, in particular $\mdim(F)=0$. But $f$ being invertible, we have  
	$\mdim( F)=1$ by Theorem \ref{strong}. When the North-South dynamic $f$ is moreover assumed to be smooth, the non-wandering set is the whole set $(\mathbb S^1)^{\mathbb Z}$. Indeed if $\mu$ is any probability measure on $\mathbb S^1$,  one easily checks that  $\prod_{k\in \mathbb Z} f^{-k}\mu$ is invariant by $F$, in particular its support is contained in the non-wandering set. If $\mu$ is the Lebesgue measure,  $\prod f^{-k}\mu$ has full support in $(\mathbb S^1)^{\mathbb Z}$ for a diffeomorphism $f$.

	For a unit CA $T_f$, we have $X_\infty=\bigcap_nf^n(X)$. Moreover the restriction of $F$ to $X_\infty^{\mathbb Z}$ is permutative. In particular it has infinite topological entropy if and only if $\sharp X_\infty=+\infty$ by Lemma \ref{entropy}. 
	If $X_\infty$ is finite then there are finitely many periodic orbits attracting all the points of $X$. When $X$ is connected so is $X_\infty$. Therefore in this case $\sharp X_\infty<\infty$ is equivalent to $(X,f)$ has an attracted fixed point with full basin.
	
\subsection{Upper bound on the mean dimension}		
		
\begin{prop}\label{inequnit}
Let $T_f$ be a unit  CA  with local rule $f:X\circlearrowleft$.
Then we have $$\mdim(X^\Z, T_f)\leq \stab(\widetilde{X_f}).$$
\end{prop}	

	\begin{proof}
	
We first prove $\mdim(T_f)\leq \dim(\widetilde{X_f})$. It is enough to consider the unilateral action $T_f^+$ on $X^\N$. To simplify the notations we write here $T_f$  for $T_f^+$.	Let $\mathcal U$ be an open cover of $X$ and $\mathcal V=\mathcal V(\mathcal U, M)$ be the induced cover of $X^\mathbb N$ given by open sets of the form $\{(x_k)_{k\in \N}\in X^\mathbb{N}: \ x_l\in U_l \text{ for  }0\le l\leq M\}$ with some $U_0,\cdots, U_M\in \mathcal U$. To conclude it is enough to prove
\begin{equation}
\lim_{n\to \infty}\frac{D(\bigvee_{k=0}^{n-1} T_f^{-k}\mathcal V )}{n}\leq \dim(\widetilde{X_f})
\end{equation}
for open covers $\mathcal V$ of the previous form, because they may have arbitrarily small diameter. 

We identify $\widetilde{X_f}$ as a closed subspace of $X^\N$. We consider an open  cover $\widetilde{\mathcal V}$ of $\widetilde{X_f}$ finer than $\mathcal V\cap \widetilde{X_f}$ with $\text{ord}(\widetilde{\mathcal V})\leq\dim(\widetilde{X_f})$.  For a cover $\mathcal A$, we let $\text{cl}(\mathcal A)$  be the associated  closed cover $\text{cl}(\mathcal A):=  \{\overline{A}: \ A\in \mathcal A\}$. By Corollary 1.6.4 in \cite{coornaert2015topological} we may  assume $\text{ord}(\text{cl}(\widetilde{\mathcal V}))=\text{ord}\left( \widetilde{\mathcal V} \right)$ without loss of generality. Any $W\in \widetilde{\mathcal V}$ may be written as $W=O_W\cap \widetilde{X_f}$ for some open subset $O_W$ of $X^\mathbb N$ in such a way that the family $\mathcal W=\{O_W: \ W\in \widetilde{\mathcal V}\}$ is finer than $\mathcal V$ and satisfies $\text{ord}(\text{cl}(\mathcal W))\leq \text{ord}(\widetilde{\mathcal V})$. Indeed if this last condition could not be satisfied, there would be a point $x\in \widetilde{X_f}$ with $\sum_{W\in \mathcal V}1_{\overline{W}}(x)-1>\text{ord}(\widetilde{\mathcal V})= \text{ord}(\text{cl}(\widetilde{\mathcal V}))$. Then by letting  $N$ large enough, we may ensure that :
\begin{itemize}
\item  $\text{ord}(\text{cl}(\mathcal W_N))\leq \text{ord}(\text{cl}(\mathcal W))\leq \text{ord}(\widetilde{\mathcal V})$ where $\mathcal W_N$  is the family of open subsets  of $X^N$ consisting in the $N$-first coordinates projections $\pi_{N}(O_W)$, $W\in \widetilde{\mathcal V}$ ;
\item  $\mathcal W$ 
is  covering the compact set $X_N:=\{(x_k)_{k\in \N}\in X^\N: \ f(x_{k+1})=x_{k} \text{ for }k=0,\cdots, N-1 \}$ since we have $\widetilde{X_f}=\bigcap_{n\in \N}X_n$. 
\end{itemize}
 Fix such a $N>M$.  Let $\mathcal X_n$ be the cover of $X^\mathbb N$ given by the sets $$F_E:=\{(x_k)_{k\in \N}\in X^\N: \ (f^{l}x_l,\cdots, f^{l-N+1}x_l)\in E_l, \ N\leq l<n\}$$ for $E_{N}, \cdots, E_{n-1}\in \mathcal{W}_N$.  Clearly $\text{ord}(\mathcal X_n)\leq (n-N)\text{ord}(\mathcal W_N)\leq (n-N)\dim(\widetilde{X_f})$. We check now $\mathcal X_n$ is finer than $\bigvee_{k=0}^{n-M}T_f^{-k}\mathcal V$, which will  imply $\mdim(T_f)\leq \dim(\widetilde{X_f})$. Let $E_l$, $N\leq l<n$, be in $\mathcal W_N$ and let $F_E$ be the corresponding element of $\mathcal X_n$. We will show that $F_E$ is contained in some element of $\bigvee_{k=0}^{n-M}T_f^{-k}\mathcal V$, i.e. for any $0\leq k<n-M$ and $0\leq l\leq M$ there is $U\in \mathcal U$ such that $f^k x_{k+l}$ lies in $U$ for all $x\in F_E$.  Note that  $(f^ix_{k+l})_{ i=k+l,\cdots,k+ l-N+1}$ lies in some $E_{k+l}\in \mathcal W_N$. Since $\mathcal{W}\succ \mathcal V$, there is $U\in \mathcal U$ such that $f^{k}x_{k+l}$ lies in $U$ for all $x\in F_E$. 

Let us show now $\mdim(T_f)\leq \stab(\widetilde{X_f})$. For $n\in \mathbb N^*$ we consider the direct $n$-product $T_f^{\times n}:=\underbrace{T_f\times \cdots \times T_f}_{n \text{ times}}$. This product is conjugated to $T_{f^{\times n}}$ with $f^{\times n}:X^n\circlearrowleft$, $(x_1,\cdots, x_n)\mapsto (f(x_1),\cdots , f(x_n))$. Moreover the natural extension of $f^{\times n}$ is conjugated to the direct $n$-product of 
$(\widetilde{X_f}, \widetilde{T_f})$, in particular $\dim(\widetilde{X_{f^{\times n} }})=\dim(\widetilde{X_f}^n)$. Therefore we get 
\begin{align*}
\mdim(T_f)&\leq \lim_n\frac{\mdim(T_f^{\times n})}{n},\\
&\leq \lim_n\frac{\mdim(T_{f^{\times n}})}{n},\\
&\leq \lim_n\frac{\dim(\widetilde{X_{f^{\times n} } })}{n},\\
&\leq \lim_n\frac{\dim(\widetilde{X_{f }}^n)}{n}= \stab(\widetilde{X_f}).
\end{align*}
This completes the proof. 

	\end{proof}

	\subsection{Natural Extension }	
The natural extension of a unit CA is again a unit CA. More precisely we  have :

\begin{prop}\label{natunit}
Let $f:X\circlearrowleft$ be a topological dynamical system and let $T_f$ be the associated unit cellular automaton. Then the natural extension   $(\widetilde{X^{\Z}_{T_f}}, \widetilde{T_f})$ is topologically conjugated to $(\widetilde{X_f}^\Z, T_{\widetilde{f}})$.
\end{prop}

\begin{proof}
Here we denote $Y_{\mathbb Z}=Y^{\mathbb Z}$ for any set $Y$. The natural extension $\widetilde{X_{T_f} }$
 is given by $$\widetilde{X_\Z}:=\widetilde{X^{\Z}_{T_f}}=\{(x^l)_{l\in \N}\in (X_\mathbb{Z})^{\mathbb N}, \ \forall l \  T_f(x^{l+1})=x^{l}\}.$$  With  $x^l=(x^l_k)_k$, the equality $T_f(x^{l+1})=x^{l}$ may be rewritten as 
 $f(x^{l+1}_{k+1})=x^{l}_k$ for all $k\in \Z$.  On the other hand  the natural extension $\widetilde{X_f}$   of $(X,f)$ is defined as 
 $$ \widetilde{X}= \widetilde{X_f}:=\{(x^l)_{l\in \N}\in X^\N, \ \forall l \ f(x^{l+1})=x^{l}\}. $$
We consider the map 
 \begin{align*}\pi : & \widetilde{X}_\mathbb{Z}\rightarrow \widetilde{X_\mathbb Z},\\
 & (\tilde{x}_k)_k\mapsto (X^l)_l  
\end{align*}
with $\tilde{x}_k=(x_k^l)_l\in \widetilde{X}$ and $X^l=(x^l_{k-l})_k\in X_{\mathbb Z}$. 
This maps takes value in $\widetilde{X_\mathbb Z}$ because 
\begin{align*}
(T_f(X^{l+1}))_k=f(x_{k-l}^{l+1})= x_{k-l}^{l} \text{ and consequently }
T_f(X^{l+1})=X^{l}.
\end{align*}
One shows also easily that $\pi$ is bijective and continuous. For $k\in \Z$ we let $x^{-1}_k=f(x^0_k)$. Let us check now $\pi$ is a conjugacy :
\begin{align*}
\pi\circ T_{\widetilde{f}}\left((\tilde{x}_k)_k\right)&=\pi\left( (\widetilde{f}(\tilde{x}_{k+1})_k\right),\\
&=\pi \left( (x^{l-1}_{k+1})_{l,k} \right),\\
&=(x_{k-l+1}^{l-1})_{l,k}.
\end{align*}
and 
\begin{align*}
\widetilde{T_f}\circ \pi\left( (\tilde x_k)_k\right)&=\widetilde{T_f}\left((x^l_{k-l})_{l,k}\right),\\
&=(x^{l-1}_{k-l+1})_{l,k}.
\end{align*}
This completes the proof.
\end{proof}	

\begin{rem}
	The natural extension of a general  is a generalized subshift of finite type. The proof is presented in Appendix \ref{app:B}.
\end{rem}	

\subsection{Mean dimension of unit CA}	
The system 
$\left( \left (\widetilde{X_f}\right)^{\Z}, T_{\widetilde{f}} \right)$ is topologically conjugated to the shift on $\widetilde{X_f}$. From Proposition \ref{natunit}, Proposition \ref{inequnit} and  Proposition \ref{prop:upper bound} we derive the following formula for the mean dimension of unit CA's.
	\begin{thm}
Let $F=T_f$ be a unit CA associated to a topological system $(X,f)$. Then we have
$$\mdim(X^\mathbb Z, T_f)=\mdim(\widetilde{X^\mathbb Z_{T_f}}, \widetilde{T_f})=\stab(\widetilde{X_f}).$$ 
\end{thm}

By considering the unit CA associated to the topological systems given by 	Proposition \ref{examples} we get :
	\begin{cor}
		For any integers $0\leq k\leq n$, there exists a permutative unit cellular automaton $F$ on $X^\mathbb{Z}$ with  $ \mdim(F)=k$ and $\stab(X)=n$. For $k\geq 1 $ we may assume $X$ connected.
	\end{cor}

\begin{rem}
When $f$ is a near homeomorphism on $X$, i.e. the uniform limit of homeomorphisms of $X$, then $T_f$ is a near strongly permutative CA and therefore $\mdim(T_f)=\stab(X)=\stab(\widetilde{X_f})$. This last inequality follows in this case from the stronger fact that $X$ and $\widetilde{X_f}$ are homeomorphic  for near homeomorphisms \cite{morton} .
\end{rem}	

	\section{Algebraic CA}

	In this section, we investigate the cellular automaton with algebraic structure. Firstly, we show that one-dimensional algebraic permutative CA has maximal mean dimension. Secondly, we prove that higher dimensional algebraic permutative CA has infinite mean dimension. Finally, building on an example of Meyerovitch \cite{meyerovitch2008finite}, we give an example of multidimensional algebraic surjective CA with positive finite mean dimension. 
	
	
	 A cellular automaton $F$ on $X^{\mathbb{Z}^d}$ with local rule $f:X^I\rightarrow X$ for $I \subset \Z^d$ is said to be {\it algebraic} if $X$ is a compact metrizable abelian group and $f$ is a continuous group homomorphism.
	
\subsection{One-dimensional algebraic permutative CA}
An \textit{algebraic} system $(X,T)$ is a topological system such that 	$X$ is a compact metrizable abelian group and $T$ is a continuous group homomorphism on $X$.
A topological extension $\psi:(Y,S)\rightarrow (X,T)$ between two algebraic systems $(Y,S)$ and $(X,T)$ is said algebraic when $\psi:Y\rightarrow X$ is a homomorphism of group.  We recall in the next two Lemmas the algebraic structure  of the natural extension of an algebraic dynamical system. The easy proofs are left to the reader.

	\begin{lem}\label{lem:algebraic nat}
	Let $(X,T)$ be an algebraic dynamical system.  Then the natural extension $(\widetilde{X_T}, \widetilde{T})$ is also algebraic.
	\end{lem}

		\begin{lem}\label{lem:algebraicnatextension}
	Let $(Y,S)$ and $(X,T)$ be  algebraic dynamical systems with $(Y,S)$ being invertible. Assume $\psi:(Y,S)\rightarrow (X,T)$  is an algebraic extension. Then the induced topological extension $\widetilde{\psi}:(Y,S)\rightarrow (\widetilde{X_T}, \widetilde{T})$ is also algebraic. 	
	\end{lem}


	Now we deduce the formula of mean dimension of algebraic permutative CA.
	\begin{lem}\label{lem:algebraic}
		Let $F$ be an algebraic permutative cellular automaton on $X^\Z$. Then $\mdim(X^{\Z}, F)=\stab(X) \cdot \text{\rm diam}(I\cup\{0\}).$ 
	\end{lem}
	\begin{proof}
	
		By Proposition \ref{prop:upper bound}, it is sufficient to show $\mdim(X^\Z,F)\ge \stab(X) \cdot \sharp J^*,$ where $J^*=J\setminus\{0\}$ and $J$ is the collection of integers in the convex hull of $I$.
		Let $(\widetilde{X^\Z}, \widetilde{F})$ be the natural extension of $(X^\Z, F)$. 
		Since $F$ is permutative and the natural extension of $(X^\N, \sigma^{J^*})$ is $(X^\Z, \sigma^{J^*})$, by Subsection \ref{previ} we have the following commutative graph: 
	\begin{equation*}
	\xymatrix{
		(X^\Z, F)  \ar[d]   &    \ar[l] (\widetilde{X^\Z}, \widetilde{F}) \ar[dd]_{\widetilde{\psi}} \ar[ldd]_{\psi}  \\
		(X\times X^\N, g)\ar[d] & \\
		(X^\N, \sigma^{J^*})  & \ar[l]_\pi (X^\Z, \sigma^{J^*}). 
	}
	\end{equation*}
By Lemma \ref{lem:algebraic nat} and Lemma \ref{lem:algebraicnatextension}, the invertible  system  $(\widetilde{X^\Z_F}, \widetilde{F})$ is  algebraic  and $\widetilde{\psi}:(\widetilde{X^\Z_F}, \widetilde{F})\rightarrow  (X^\Z, \sigma^{J^*})$ is an algebraic extension.
	By \cite[Corollary 6.1]{li2018mean}, we have $\mdim(\widetilde{X^\Z}, \widetilde{F})\ge \mdim(X^\Z, \sigma^{J^*})=\stab(X) \cdot \sharp J^*.$ Meanwhile, by Proposition \ref{nat}, we obtain
	$$
	\mdim(X^\Z,F)\ge \mdim(\widetilde{X^\Z_F}, \widetilde{F}) \ge \stab(X) \cdot \sharp (J^*).
	$$
	\end{proof}
	
	\begin{ex}
		Let $f: \T^{\{0,1\}}\to \T$ defined by $(x,y)\mapsto 2x+3y$. The associated  cellular automaton $T_f$ on $\T^\Z$ is an algebraic permutative CA. By Lemma \ref{lem:algebraic}, we have $\mdim(\T^\Z, T_f)=1$.
	\end{ex}
	
\subsection{Higher dimensional algebraic permutative CA }	
	
Let $d>1$.
Following the notations used in \cite{Bur20} we  let $\mathbb I$ be the convex hull of the domain $I$ of a CA on $X^{\mathbb Z^d}$. The {\it support function}  of $\mathbb I$ is the function  $h_{\mathbb I}:\mathbb S^{d-1}\rightarrow \mathbb R$, which maps $u\in \mathbb S^{d-1}$ to $\max_{i\in \mathbb I} i\cdot u$ with $\cdot$ be the usual scalar product on $\R^d$.
For a convex $d$-polytope $J$ in $\mathbb R^d$, a face $F$ of $J$  and $\epsilon\in \mathbb R$ we denote by $N^F$ the exterior normal vector to $F$ and by $T_F^+J(\epsilon)$ the  closed semi-space normal to $N^F$   satisfying  $[ F+\epsilon'N^F\subset T_F^+J(\epsilon)]\Leftrightarrow \epsilon'\geq \epsilon $.

The $\mathbb I$-morphological boundary of $J$ is the subset 
$$\partial_{\mathbb I}^-J:=\{j\in J, \, i+j\in J \text{ for all }i\in \mathbb I\}.$$
We  consider the subset of $\partial_{\mathbb I}^-J$  given by   $\partial^-_{\mathbb I}F:=\partial_{\mathbb I}^-J\cap T_F^+J(-h_{\mathbb I}(N^F))$. 
The sets $\partial^-_{\mathbb I}F$ over faces  $F$ of $J$ are covering $\partial^-_{\mathbb I}J$ but do not define a partition in general. 
For any face $F$ of $J$   we let $u^F\in ex(\mathbb I)\subset I'$ with $u^F\cdot N^F=h_{\mathbb I}(N^F)$.  We  also let $\mathcal F_{\mathbb I}(J)$ be the set of  faces $F$ for which $u_F$ is uniquely defined.  We denote by $\partial_{\mathbb I}^{\bot}J$ the subset of $\partial^-_{\mathbb I}J$ given by  $$\partial_{\mathbb I}^{\bot}J:=\bigcup_{F\in \mathcal F_{\mathbb I}(J)}\partial^-_{\mathbb I}F.$$


Let $T_f:X^{\Z^d}\circlearrowleft$, $d>1$, be a higher-dimensional CA associated to a local rule  $f:X^I\rightarrow X$ with $I\subset \mathbb Z^d$. We say that $T_f$ is \textit{permutative} when for any extreme point $j\neq 0^d$ of $\mathbb{I}$ and    for any $x^j=(x_i)_{i\in I\setminus \{j\}} \in X^{I\setminus \{j\}}$, the map   $f_{x^j}:X\circlearrowleft $,  $x_j \mapsto f((x_i)_{i\in I})$, is surjective.

For a subset $E$ of $\mathbb R^d$, we let $\underline{E}$ the set of integers in $E$. By Lemma 13 in \cite{Bur20}, the semi-conjugacy $g=g_{\underline{J} , \underline{J\setminus \partial_{\mathbb I}^{\bot}J}}$ is surjective. Moreover for any domain $I\neq \{0\}$ we may choose $J$ such that $\underline{\partial_{\mathbb I}^{\bot}J}$ has arbitrarily large cardinality (see Section 7 in \cite{Bur20}).  Recall now that  $g_{\underline{J} , \underline{J\setminus \partial_{\mathbb I}^{\bot}J}}$ semi-conjugates $F$ with the a skew-product over the full shift on $X^{\underline{\partial_{\mathbb I}^{\bot}J}}$.  Arguing as in Lemma \ref{lem:algebraic} we conclude  that  : 
	\begin{lem}\label{lem:multidimension algebraic}
		Let $F$ be an algebraic, permutative, non trivial (i.e. with $I\neq \{0\}$) cellular automaton on $X^{\Z^d}$ with $d>1$. Then $\mdim(X^{\Z^d}, F)$ is infinite. 
	\end{lem}

\subsection{Examples of higher dimensional CA's with finite nonzero mean dimension}
A {\it $\Z^d$-tiling system} is a pair $(S,R)$ where $S$ is a finite set of "square tiles" and $R \subset S^E$ is some adjacency rules  with $E=\{0, \pm e_1, \pm e_2, \dots, \pm e_d  \}$,
which determine when a tile $s_1\in S$ is allowed to be placed next to a tile $s_2\in S$ (and in which directions). A configuration $x\in S^F$ for some $F\subset \Z^d$
is called {\it valid} at $n\in F$ if the neighbours of the cell at $n$ obey the adjacency rules, i.e. $x_{F+n}\in R$. Clearly, the set of infinite valid configurations in $S^{\Z^d}$ forms a subshift of finite type.

A set of {\it directed tiles} $S$ with direction $d$ is defined by a tiling system with a forward direction $d(s) \in \{\pm e_1, \dots, \pm e_d\}$ associated to each tile $s\in S$. Given a configuration $x\in S^{\Z^d}$, a {\it path} defined by $x$ is a sequence $p_0, p_1, p_2, \dots$ with  $p_0$ given and $p_{n+1}\in \Z^d, n\in \N,$ obtained by traversing the forward directions of $x$, that is to say, $p_{n+1} = p_n + d(x_{p_n})$. Given $x\in S^{\Z^d}$
, a path $\{p_n\}_{n\in \N}$ is {\it valid} if $x$ is valid at every $p_n$. A set of directed tiles $S$ is called an {\it acyclic} set of tiles if any valid path in $x\in S^{\Z^d}$ is not a loop.

Let $S$ be a set of directed set of tiles. Let $X:=(S\times \T)^{\Z^d}$.  Define $T: X\to X$ by
\begin{equation*}
T(x,y)_n=
\begin{cases}
(x_n, y_n+y_{n+d(x_n)})~&\text{if}~x~\text{is valid at}~n,\\
(x_n, y_n)~&\text{otherwise}.
\end{cases}
\end{equation*}

Notice that if $S$ is an acyclic set, then $T$ is surjective. 
Now suppose $S$ is an acyclic set. Let $\omega\in S^{\Z^d}$ and 
$$
X_\omega:=\{(s,y)\in X: s=\omega \}
$$
which is a closed $T$-invariant subset of $X$. Now define a directed graph $G_\omega=(V_\omega, E_\omega)$ with the vertex $V_\omega=\Z^d$ and the edges
$$
E_\omega=\{(n, n+d(\omega_n)): n\in \Z^d, \omega~\text{is valide at}~n \}.
$$
A subset of vertex $K\subset V_\omega$ is said to be {\it connected} in the graph $G_\omega$ if for any $a,b\in K$ there exists  $c\in K$ such that there are directed paths in $K$ from $a$ to $c$ and from $b$ to $c$. A {\it connected component} of $G_\omega$  is a connected set in $G_\omega$ which is maximal with
respect to inclusion. A connected component $K$ of $G_\omega$ is called {\it forward-infinite} if there exists a forward-infinite directed graph in $G_\omega$ starting at some/any cell of $K$. Note that an infinite connected component is not always forward-infinite, in other words, it may be backward-infinite.
Since $K$ is acyclic, the set of vertex $V_\omega$ is a disjoint union of connected components of $G_\omega$. 

By definition of $T$, it is clear that $T$ acts independently on each connected component of $G_\omega$. Suppose $G_\omega$ has $m$ forward-infinite connected component, denoted respectively by $K_1, \dots, K_m$. Let $K_0=V_\omega \setminus \sqcup_{i=1}^m K_i$ which is the union of vertex having a forward-finite path. Let $(X_i, T_i)$ be the system corresponding to the action of $T$ on $K_i$. Clearly, $(X_\omega, T)=\prod_{i=0}^{m} (X_i, T_i)$.

\begin{lem}\label{lem:omega mean dimension}
	We have $\mdim(X_0,T_0)=0$ and $\mdim(X_i, T_i)=1$ for every $1\le i\le m$. Moreover,  $1\le \mdim(X_\omega, T)\le m$ whenever $m\ge 1$.
\end{lem}
\begin{proof}
	Since any valid path in $X_0$ is forward-finite, the system $(X_0, T_0)$ is isomorphic to an inverse limit of finite-dimensional systems: the $k$-th system in this sequence consists of the cells in $K_0$ with forward-path in $G_\omega$ of length at most $k$, which has the topological dimension at most $k$\footnote{Because it is a countable union of spaces having dimension at most $k$.}. Since a finite-dimensional system has zero mean dimension, it follows by \cite[Proposition 5.8]{shi2021marker} that $\mdim(X_0,T_0)=0$.
	
	It remains to show $\mdim(X_i, T_i)=1$ for every $1\le i\le m$. Let $K=K_i$ be a forward-infinite connected component in $G_\omega$.  We define inductively a sequence $J_n\subset K$ as follows. Let $J_0=\{p_0, p_1, p_2, \dots \} \subset \Z^d$ be a forward-infinite path in $K$. If $J_n=K$, then let $J_{n+1}=J_n$; otherwise pick a cell $a\in K\setminus J_n$ whose successor in $J_n$ and let $J_{n+1}=J_n\cup \{a\}$. Then we have 
	$$
	J_0\subset J_1 \subset \dots  \subset J_n \subset \dots \subset K~\text{and thus}~ K=\bigcup_{n\ge 0}J_n.
	$$
	Notice that $T$ acts independently on each $J_n$. It is clear that the action of $T$ on $J_0$ is isomorphic to the algebraic CA on $\T^\N$ of the form $(x_n)_{n\in \N} \mapsto (x_n+x_{n+1})_{n\in \N}$. It follows by Lemma \ref{lem:algebraic} that  the action of $T$ on $J_0$ has the mean dimension $1$. Since the action of $T$ on $J_n$ is a skew-product extension of that on $J_{n-1}$, the action $T$ on $K$ is an inverse limit of these systems. By \cite[Proposition 5.8]{shi2021marker}, we have $\mdim(X_i, T_i)\le 1$. On the other hand, since the action $T$ on $K$ is a  skew-product extension of that on $J_0$, by  Lemma \ref{lem:skew product}, we have $\mdim(X_i, T_i)\ge 1$. Thus we have $\mdim(X_i, T_i)=1$.
	
	By mean dimension of product systems, we conclude that $1\le \mdim(X_\omega, T)\le m$ whenever $m\ge 1$.
\end{proof}

For each $\omega\in S^{\Z^d}$, define $I(\omega)$ to be the number of forward-infinite connected components of $G_\omega$. For a directed set of tiles $S$, let $I(S)=\sup_{\omega\in S^{\Z^d}} I(\omega)$. 

\begin{prop}\label{prop:example finite nonzero higher dimension}
	For any $d\ge 1$ there exist a surjective (algebraic) $\Z^d$-CA with positive, finite mean dimension.
\end{prop}	
\begin{proof}
	
	Let $d\ge 1$. By \cite[Section 4]{meyerovitch2008finite}, there exists a directed set of tiles $S$ which is an acyclic set and has $0<I(S)<+\infty$. It follows from Lemma \ref{lem:omega mean dimension} and $\mdim(X,T)=\sup_{\omega\in S^{\Z^d}} \mdim(X_\omega, T)$ that
	$$
	1\le \mdim(X,T)\le I(S)<+\infty.
	$$
	This completes the proof.
\end{proof}	
	
Finally, we remark that by Lemma \ref{lem:multidimension algebraic} the CA presented in Proposition \ref{prop:example finite nonzero higher dimension} for $d>1$ is not permutative.

	\section{Higher-dimensional CA having a spaceship}\label{sec:multi}

Let $X$ be a compact metric space.  For $v\in \Z^d$, let $\sigma_v: X^{\Z^d}\to X^{\Z^d}$ be the shift $(x_u)_{u\in \Z^d}\mapsto (x_{u+v})_{u\in \Z^d}$. 
	Let $Y$ (resp. $x_*$) be a distinguished subset (resp. point) of $X$. 	Let $T$ be another compact metric space with $\stab(T)>0$ and $(h_t:X\circlearrowleft)_{t\in T}$ a family of functions on $X$, such that 
	\begin{itemize}
		\item the map $t\mapsto h_t(x)$ is continuous for $x\in X$,
		\item the map $t\mapsto h_t(x)$ is injective for $x\in Y$,
		\item $h_t(x_*)=x_*$ for all $t$.
	\end{itemize}
	To simplify the notations we also let $h_t(x)=t \cdot x$ and $h_t\circ h_{t'}(x)=tt'\cdot x$.


Let $F$ be  a  CA on $X^{\Z^d}$ associated to a local rule $f:X^{I}\rightarrow X$ for a finite subset $I$ of $\mathbb Z^d$ satisfying $f(t\cdot y_i, \ i\in I)=t\cdot f(y_i, \ i\in I)$ for all $(y_i)_{i\in I}\in (Y\cup\{x_*\}) ^I$ and for all $t\in T$.

	The {\it support} of an element $x=(x_u)_{u\in \Z^d}\in  X^{\Z^d}$ is the set
	$$
	\supp (x)=\{u\in \Z^d: x_u\neq x_* \}\subset \Z^d.
	$$
 An  element $ x\in (Y\cup\{x_*\})^{\Z^d}$ with finite non empty support is called a {\it spaceship } when $F^px=a\cdot \sigma_v(x):=(a\cdot x_{u+v})_{u\in \mathbb Z^d}$  for some   $p\in \N^*$, $v\in \mathbb Z^d\setminus\{0\}$ and $a\in T$ satisfying $ a\cdot Y\subset Y$.

	\begin{prop}\label{thm:spaceship}
		The mean dimension of   a CA having a spaceship  is infinite. 
	\end{prop}

	\begin{proof}
		Let  $F$ be a CA having a spaceship.
		 By assumption, let $x\in (Y\cup\{x_*\})^{\Z^d}$ be a spaceship of period $p$ with displacement vector $v\neq 0$, i.e. $F^px=a\cdot \sigma_v(x)$ for some $a\in T$ with $a\cdot Y\subset Y$. By considering some iterate of $F$ we may assume without loss of generality  $p=1$. 
		
		We can pick a vector $u\in \Z^d$   not proportional to $v$  and a positive integer $m$ such that the sets $\supp(x)-imv+ju-\mathbb I$ for $i,j\in \Z$ are pairwise disjoint (recall $\mathbb I$ denotes the convex hull of $I\cup \{0\}$). Let $n$ be a positive integer. The direct product of $n$ copies of the right shift $(T^\Z, \sigma^{-1})$ is denoted by $( (T^\Z)^n, \sigma^{-\otimes n})$ which has  mean dimension equal to $n\cdot \stab(T)>0$.  Define $\varphi: (T^{\Z})^n \to X^{\Z^d}$ for all $\alpha=(\alpha^1, \cdots, \alpha^n)\in (T^{\Z})^n$ by

		$$\varphi(\alpha)_{k}= \alpha_i^j a^{im}\cdot x_{k'} \text{ for }k=k'-imv+ju \text{ with }k'\in \supp(x), \ i\in \Z, \ 1\leq j \leq n$$
		
		$$\text{ and }\varphi(\alpha)_k=x_* \text{ for others }k\in \Z^d .$$
 It is not hard to see that the map $\varphi$ is continuous. As $t\mapsto t\cdot y$  is injective for $y\in Y$, the map $\varphi$ is one-to-one. Thus the map $\varphi$ is a homeomorphism from $(T^\Z)^n$ to its image $\Omega:=\varphi\left((T^\Z)^n\right)$.
		
		\vspace{5pt}
		
		\textbf{Claim}:  The dynamical system $((T^{\Z})^n, \sigma^{-\otimes n})$ is topologically conjugate to $(\Omega, F^{m})$.
		
		\vspace{5pt}
		
		Suppose our claim holds. Then $\mdim{(\Omega, F^m) }= \mdim{((T^{\Z})^n, \sigma^{-\otimes n}) }=n \cdot \stab(T)$. Thus
		$$
		\mdim{(X^{\Z^d}, F) }=\frac{\mdim{(X^{\Z^d}, F^m) }}{m}\ge\frac{\mdim{(\Omega, F^m) }}{m}=\frac{n\cdot \stab(T)}{m}.
		$$
	This will complete the proof as it holds for all positive integers $n$.
		It remains to prove our claim. In fact, it is sufficient to show $F\circ \varphi=\varphi\circ \sigma^{\otimes n}$.
		
		For $k\in \Z^d$ the $k$-coordinate of $F(\varphi (\alpha))$ depends only on the $l$-coordinates of $\varphi(\alpha)$ for $l\in k+ \mathbb I$. Fix $k$. There is at most one  pair $(i,j)\in\Z\times \{1,\cdots, n\}$ with $\left(\supp(x)-imv+ju\right)\cap \left(k+\mathbb I\right)\neq \emptyset $,   because  the sets $\supp(x)-i'mv+j'u-\mathbb I$ for $i',j'\in \Z$ are pairwise disjoint. 
		As $t\cdot x_*=x_*$ for all $t\in T$, the configuration $\varphi (\alpha)$ coincides with $\alpha_i^ja^i \cdot\sigma_{imv-ju: }x$ on $k+\mathbb I$.  Therefore 
		
		\begin{align*}
		F(\varphi (\alpha))_k&=f(\alpha_i^ja^{im} \cdot x_{q+k+imv-ju}, \ q\in I), \\
		&		= \alpha_i^ja^{im}\cdot f(x_{q+k+imv-ju)}, \ q\in I),\\
		&= \alpha_i^ja^{im}\cdot F(x)_{k+imv-ju},\\
		&= \alpha_i^ja^{im+1}\cdot x_{k+(im+1)v-ju}.
			\end{align*}
		
	Iterating again $F$ we get finally 
	\begin{align*}
		F^m(\varphi (\alpha))_k&= \alpha_i^ja^{(i+1)m}\cdot x_{k+(i+1)mv-ju},\\
		&= \phi( \sigma^{-\otimes n}\alpha)_k.
			\end{align*}

	Therefore 	$ F^m(\varphi (\alpha)= \phi( \sigma^{-\otimes n}\alpha) $ for all  $\alpha=(\alpha^1, \cdots, \alpha^n)\in (T^\Z)^n$. This completes the proof of our claim.
	\end{proof}
	
We illustrate Proposition \ref{thm:spaceship} 	with a continuous state version of the celebrated 
	Conway's game of life. We first recall the local rule of this famous CA on $\{0,1\}^{\Z^2}$. Let $I=[-1,1]^2\cap \Z^2$ and $I^*=I\setminus \{0\}$.  The local rule of the game of life is the map $f:\{0,1\}^I\rightarrow \{0,1\}$ such that $f(x_i, i\in I)=1$ if and only if either $x_0=1$ and $\sharp \{i\in I^*: \ x_i=1\}\in \{2,3\}$ or $x_0=0$ and   $\sharp \{i\in I^*: \ x_i=1\}=3$. It is well  known that this discrete CA has a spaceship, meaning here that  there a  finitely supported configuration $x\in \{0,1\}^{\Z^2}$ (i.e. with finitely many non-zero coordinates)   satisfying $F(x)=\sigma_v(x)$ for some $ v \in \mathbb Z^2\setminus \{0\}$.
	
	We describe below a continuous version of the game of life on $[0,1]^{\Z^2}$ which contains the standard discrete version as a subsystem.  We let $f$ be the local rule $f:[0,1]^I \rightarrow [0,1]$ defined as follows :
	
\begin{itemize}
\item if $x_0>0$ and $\sharp \{i\in I^*: \ x_i>0\}\in \{2,3\}$ we let $f(x_i)=\frac{\sum_{i\in I^*}x_i}{\sharp \{i\in I^*: \ x_i>0\} }$,
\item if $x_0=0$ and $\sharp \{i\in I^*: \ x_i>0\}=3$ we let $f(x_i)=\frac{\sum_{i\in I^*}x_i}{3}$,
\item in the remaining case, we let $f(x_i, \ i\in I)=0$. 
\end{itemize}	

\begin{cor}
The continuous state version of the game of life has infinite mean dimension. 	
	\end{cor}
	
\begin{proof}
Apply Proposition \ref{thm:spaceship} with $X=[0,1]$, $Y=]0,1]$, $x_*=0$, $T=[0,1]$, $a=1$  and 
$h_t(x)=tx$ for all $t,x$. 
\end{proof}

\appendix

\section{Zero metric mean dimension}\label{app:A}

	Lindenstrauss and Weiss \cite{LindenstraussWeiss2000MeanTopologicalDimension} showed that the metric mean dimension is always an upper bound of the topological mean dimension. For general dynamical systems, it is widely open whether there exists a metric $d$ (compatible with the topology) such that the metric mean dimension in terms of $d$ equals the mean dimension.  This would be  a dynamical version of Pontrjagin-Schnirelmann's theorem. 
	In this appendix, we show that if a dynamical system $(X,T)$ is an inverse limit of dynamical systems of finite topological entropy, then there exists a metric $d$ on $X$ with $\mdim_M(X,T,d)=0$. 
	
	\begin{prop}\label{thm:inverse limit metric}
		Let $(X,T)$ be the inverse limit of a sequence of dynamical systems of finite topological entropy. Then there exists a metric $d$ such that $\mdim(X,T)=\mdim_M(X,T,d)=0$. 
	\end{prop}	
	\begin{proof}
		Let $(X,T)=\varprojlim_k (X_k, T_k)$ with $(X_k, T_k)$ of finite topological entropy.   Pick a metric $\rho^k$ on $X_k$ for each $k$ and define a metric $\hat{\rho}^k(x_k,x_k'):=\max_{0\leq j\leq k}\rho^j(x_j,x_j')$ on $X_k$ for each $k\ge 0$. Then for each $0\leq j<k$ the bonding map  $x_k\mapsto x_j$  from $(X_k, \hat{\rho}^k)$ to $(X_j, \hat{\rho}^j)$ is $1$-Lipschitz. Moreover it follows from $h_{top}(X_k,T_k)<+\infty$ that $\mdim_M(X_k, T_k, \hat{\rho}^k)=0$ for each $k$. 
	By Lemma 7.7 in \cite{li2018mean} there is a  distance $d$ on $X$ satisfying 
	$$ \mdim_M(X,T,d)\leq \liminf_k  \mdim_M(X_k,T_k,\hat{\rho}^k )=0.$$

	\end{proof}
	
	
	\begin{cor}
		If $(X,T)$ is finite dimensional, has at most countably many ergodic measures or has the small boundary property, then there exists a metric $d$ such that $\mdim(X,T)=\mdim_M(X,T,d)=0$.  
	\end{cor}	
	\begin{proof}
		Due to \cite{SW}, \cite{L95} and \cite{L99}, if $(X,T)$ is finite dimensional, has at most countably many ergodic measures or has the small boundary property, then $(X,T)$ is the inverse limit of a sequence of dynamical systems with finite topological entropy, then the result follows by Proposition \ref{thm:inverse limit metric}.
	\end{proof}	
	
	

\section{The natural extension of general cellular automata}\label{app:B}

In Section \ref{sec:Unit CA}, we show that the natural extension of a unit CA is a full shift. In this section, we prove that for a general CA, its natural extension is a subshift of finite type. 
Recall that a subshift $(Y,\sigma)$ with $Y\subset X^{\mathbb Z}$ is said of {\it finite type} when there is a closed subset $L$ of $X\times X$ with  
$$Y=\{(x_n)_n\in X^{\mathbb Z}, \ (x_n,x_{n+1}) \in L \text{ for all }n\in  \mathbb Z\}.$$

\begin{prop}\label{prop:sft}
	Let  $T_f$ be a cellular automaton on $X^{\mathbb{Z}}$ with  local rule $f:X^I\rightarrow X$ for $I \subset \Z^d$. The natural extension   $\widetilde{T_f}$ is topologically conjugated to a subshift of finite type.
\end{prop}	

\begin{proof}
	Here we denote $Y_{\mathbb Z}=Y^{\mathbb Z}$ for any set $Y$.  Let $J$ be the integers in the convex hull of $I\cup \{0\}$. Let   $J_-=\min\{J\}$ and $J_+=\max\{J\}$. Let $nJ$ be the collection of integers in $[nJ_-, nJ_+]$.
	The natural extension $\widetilde{X_\mathbb{Z}}$
	is given by $$\widetilde{X_\mathbb{Z}}=\{x^l\in X_{\mathbb Z}, \ T_f(x^{l+1})=x^{l}\}\subset (X_\mathbb{Z})^{\mathbb N}.$$  With  $x^l=(x^l_k)_k$, the equality $T_f(x^{l+1})=x^{l}$ may be rewritten as 
	$f((x^{l+1}_k)_{J_-+j\le k\le J_+ +j} )=x^{l}_j$ for all $j$.

	We let 
	$$  \widetilde{X}:=\{(x^l)_l\in \prod_{l\in \N} X^{lJ}: x^l=(x^l_k)_{k\in lJ}, f((x^{l+1}_k)_{J_-+j\le k\le J_+ +j} )=x^{l}_j\}. $$
Consider the subset $L$ of  $\widetilde{X} \times \widetilde{X}$ given by
	$$
	L:=\{(x,y)\in \widetilde{X} \times \widetilde{X}: x_k^l=y_{k}^{l+1}, \forall l\ge 0, \forall ~lJ_-\le k\le lJ_+ \}.
	$$
Then we define  a subshift in $\widetilde{X}_\Z$ by 
	$$
	Y:=\{x\in\widetilde{X}_\Z: (x_n, x_{n+1})\in L  \}.
	$$
The map 
	\begin{align*}\pi:&Y\rightarrow \widetilde{X_\mathbb Z},\\
	& (y_k)_{k\in \Z}\mapsto (x_n^l)_{n\in \Z, l\in \N}  
	\end{align*}
	with $y_k=([y_k]_n^l)_{l\in \N, n\in lJ}\in \widetilde{X}$ and $x_n^l=[y_k]_n^{l+k}$ for some/any  $n\in (l+k)J$. It is clear that $\pi$ is a topological conjugacy.
	
\end{proof}	

\begin{rem}
	In general, the space $\widetilde{X}$ may be an infinite-dimensional space. Since we know few about the mean dimension of a subshift of finite type (even the full shift) over an infinite-dimensional space, the natural extension $\widetilde{T_f}$ in Proposition \ref{prop:sft} does not provide information on the mean dimension of $(X^\Z, T_f)$.
\end{rem}

	\bibliographystyle{alpha}
	\bibliography{universal_bib}

\def\cprime{$'$} \def\cprime{$'$}
\begin{thebibliography}{MW98}

\bibitem[Bro60]{morton}
Morton Brown.
\newblock Some applications of an approximation theorem for inverse limits.
\newblock {\em Proc. Amer. Math. Soc.}, 11:478--483, 1960.

\bibitem[Bur]{Bur20}
David Burguet.
\newblock Rescaled entropy of cellular automata.
\newblock {\em https://arxiv.org/abs/2005.08585}.

\bibitem[Coo05]{Coo05}
Michel Coornaert.
\newblock {\em Dimension topologique et syst\`emes dynamiques}, volume~14 of
  {\em Cours Sp\'ecialis\'es [Specialized Courses]}.
\newblock Soci\'et\'e Math\'ematique de France, Paris, 2005.

\bibitem[Coo15]{coornaert2015topological}
Michel Coornaert.
\newblock {\em Topological dimension and dynamical systems}.
\newblock Springer, 2015.

\bibitem[Eng95]{engelking1995theory}
Ryszard Engelking.
\newblock Theory of dimensions, finite and infinite.
\newblock {\em Sigma Series in Pure Mathematics}, 10, 1995.

\bibitem[Gro99]{G}
Misha Gromov.
\newblock Topological invariants of dynamical systems and spaces of holomorphic
  maps. {I}.
\newblock {\em Math. Phys. Anal. Geom.}, 2(4):323--415, 1999.

\bibitem[Gut17]{gut17}
Yonatan Gutman.
\newblock Embedding topological dynamical systems with periodic points in
  cubical shifts.
\newblock {\em Ergodic Theory Dynam. Systems}, 37(2):512--538, 2017.

\bibitem[IM12]{ingram}
W.~T. {Ingram} and William~S. {Mahavier}.
\newblock {\em {Inverse limits. From continua to chaos}}, volume~25.
\newblock Berlin: Springer, 2012.

\bibitem[Lin95]{L95}
Elon Lindenstrauss.
\newblock Lowering topological entropy.
\newblock {\em J. Anal. Math.}, 67:231--267, 1995.

\bibitem[Lin99]{L99}
Elon Lindenstrauss.
\newblock Mean dimension, small entropy factors and an embedding theorem.
\newblock {\em Inst. Hautes \'Etudes Sci. Publ. Math.}, 89(1):227--262, 1999.

\bibitem[LL04]{lakshtanov2004criterion}
Evgenii~Leonidovich Lakshtanov and Evgenii~Sergeevich Langvagen.
\newblock Criterion of infinite topological entropy for multidimensional
  cellular automata.
\newblock {\em Problems of Information Transmission}, 40(2):165--167, 2004.

\bibitem[LL18]{li2018mean}
Hanfeng Li and Bingbing Liang.
\newblock Mean dimension, mean rank, and von neumann--l{\"u}ck rank.
\newblock {\em Journal f{\"u}r die reine und angewandte Mathematik},
  2018(739):207--240, 2018.

\bibitem[LW00]{LindenstraussWeiss2000MeanTopologicalDimension}
Elon Lindenstrauss and Benjamin Weiss.
\newblock Mean topological dimension.
\newblock {\em Israel J. Math.}, 115:1--24, 2000.

\bibitem[Mey08]{meyerovitch2008finite}
Tom Meyerovitch.
\newblock Finite entropy for multidimensional cellular automata.
\newblock {\em Ergodic Theory and Dynamical Systems}, 28(4):1243--1260, 2008.

\bibitem[MW98]{morris1998entropy}
G~Morris and T~Ward.
\newblock Entropy bounds for endomorphisms commuting with k actions.
\newblock {\em Israel Journal of Mathematics}, 106(1):1--11, 1998.

\bibitem[Shi21]{shi2021marker}
Ruxi Shi.
\newblock Finite mean dimension and marker property.
\newblock {\em arXiv preprint arXiv:2102.12197}, 2021.

\bibitem[SW91]{SW}
M.~Shub and B.~Weiss.
\newblock Can one always lower topological entropy?
\newblock {\em Ergodic Theory Dynam. Systems}, 11(3):535--546, 1991.

\bibitem[Tsu19]{tsukamoto2019mean}
Masaki Tsukamoto.
\newblock Mean dimension of full shifts.
\newblock {\em Israel Journal of Mathematics}, 230(1):183--193, 2019.

\bibitem[War00]{ward2000additive}
Thomas~B Ward.
\newblock Additive cellular automata and volume growth.
\newblock {\em Entropy}, 2(3):142--167, 2000.

\end{thebibliography}

\end{document}